\documentclass[11pt,a4j]{article}
\usepackage{amsthm}

    \newtheorem{thm}{Theorem}[section]
    \newtheorem{prop}[thm]{Proposition}
    \newtheorem{lem}[thm]{Lemma}
    \newtheorem{cor}[thm]{Corollary}
    
  \theoremstyle{definition}
    \newtheorem*{que}{Question}
    \newtheorem{defi}[thm]{Definition}
    
  \theoremstyle{remark}
    \newtheorem{rem}[thm]{Remark}
    \newtheorem{ex}[thm]{Example}

\usepackage{amsmath}
\usepackage{amssymb}
\usepackage[all]{xy}
\usepackage{graphicx}
\usepackage{mathrsfs}
\usepackage{indentfirst}
\usepackage{stmaryrd}

\title{Fibration structure for Gromov h-principle}
\author{Koji Yamazaki}
\begin{document}
\maketitle
\begin{abstract}
The {\it h-principle} is a powerful tool for obtaining solutions to partial differential inequalities and partial differential equations.
Gromov discovered the h-principle for the general partial differential relations to generalize the results of Hirsch and Smale. 
In his book, Gromov generalizes his theorem and discusses the {\it sheaf theoretic h-principle}, in which an object called a {\it flexible sheaf} plays an important role. 
We show that a flexible sheaf can be interpreted as a fibrant object with respect to a model structure.
\end{abstract}
\tableofcontents	
\setcounter{section}{-1}
	\section{Introduction}
The {\it h-principle} is a powerful tool for obtaining solutions to partial differential inequalities and partial differential equations.
Hirsch \cite{hirsch1959immersions} and Smale \cite{smale1959classification1, smale1959classification2} discovered the h-principle for the immersions, and Phillips \cite{phillips1967submersions} discovered the h-principle for the submersion.
Gromov \cite{gromov1971topological} discovered {\it Gromov h-principle}, the h-principle for the general partial differential relations, to generalize the above results.
The proof of {\it Gromov h-principle} is favoured in modern times via {\it Holonomic Approximation Theorem} by Y. M. Eliashberg and N. M. Mishachev \cite{eliashberg38holonomic, eliashberg2002introduction}.
Their proof is rather elementary, as they consider a space of functions to contain the Whitney topology.
Before that, Gromov himself has refined his proofs. \par
In his book, Gromov \cite{gromov2013partial} generalizes his theorem and discusses the {\it sheaf theoretic h-principle}, in which an object called the {\it flexible sheaf} plays an important role. 
This proof is rather abstract.
For this reason, the proof by Eliashberg and Mishachev is more prefered to see the concrete images.
However, the abstract proof by Gromov also has some advantages
(cf. Section 0.5). \par
We, therefore, wish to develop further the abstract proof by Gromov himself.
We first clarify what a flexible sheaf is.
We show that a flexible sheaf can be interpreted as a fibrant object with respect to a model structure.
	\subsection{What's h-principle?}
Consider the following partial differential inequality $R$:
\begin{equation}
f(x, y, \frac{\partial y}{\partial x}, \cdots, \frac{\partial^r y}{\partial x^r}) > 0. \tag{$R$}
\end{equation}
In analysis, we often detect a weak solution and show later that it is a strict solution.
In differential topology, we detect a formal solution and transform it into a strict solution by homotopy.
{\it A formal solution} is a family of continuous maps $s_0(x), s_1(x), \cdots, s_r(x)$ such that 
\[
f(x, s_0(x), s_1(x), \cdots, s_r(x)) > 0.
\]
We say that the partial differential inequality $R$ satisfies the {\it h-principle} or the {\it homotopy principle} if any formal solution can be homotopically transformed into a strict solution.
We say that it satisfies the {\it parametric h-principle} or the {\it parametric homotopy principle} if the inclusion map $\operatorname{Sol}(R) \hookrightarrow \Gamma(R)$ is a weak homotopy equivalence where $\operatorname{Sol}(R)$ is the space of strict solutions of $R$ and $\Gamma(R)$ is the space of formal solutions of $R$.
	\subsection{Sheaf theoretic h-principle}
A {\it quasitopological space} is a generalization of a topological space, whose homotopy groups can be defined 
(cf. Section 1).
{\bf qTop} is the category of quasitopological spaces.
{\bf qTop} has a model structure 
(cf. Section 1.2).\par
Let $B$ be a topological space.
A {\it continuous sheaf} on $B$ is a sheaf on $B$ with value in {\bf qTop}.
${\bf Sh}(B; {\bf qTop})$ is the category of continuous sheaves on $B$.
For example, the spaces of solutions to a partial differential relation give a continuous sheaf.
A continuous sheaf $\mathcal{F}$ has the {\it sheaf of formal sections} $\mathcal{F}^\ast$ and the {\it diagonal map} $\Delta : \mathcal{F} \rightarrow \mathcal{F}^\ast$ 
(cf. Section 2.3).
We say that $\mathcal{F}$ satisfies the {\it (sheaf theoretic) parametric h-principle} if the $\Delta$ is a sectionwise weak equivalence.
The following fundamental theorem is the basis of Gromov theory.
\begin{thm}[\mbox{Gromov \cite[page 76]{gromov2013partial}}] \label{Gromov1}
Any flexible sheaf on a polyhedron satisfies the parametric h-principle.
\end{thm}
This theorem follows immediately from the following two lemmas.
\begin{lem}[\mbox{Gromov \cite[page 76]{gromov2013partial}}] \label{Gromov2}
Let $B$ be a polyhedron, and let $\mathcal{F}$ be a continuous sheaf on $B$.
\begin{description}
\item[(1)] The diagonal morphism $\Delta : \mathcal{F} \rightarrow \mathcal{F}^\ast$ is a stalkwise weak equivalence.
\item[(2)] The sheaf $\mathcal{F}^\ast$ of formal sections is flexible.
\end{description}
\end{lem}
\begin{lem}[\mbox{Gromov \cite[page 77]{gromov2013partial}}] \label{Gromov3}
Let $B$ be a polyhedron, and let $\mathcal{F}$ and $\mathcal{G}$ be flexble sheaves on $B$.
$f : \mathcal{F} \rightarrow \mathcal{G}$ is a sectionwise weak equivalence if $f$ is a stalkwise weak equivalence.
\end{lem}
The above Lemma \ref{Gromov2} holds in more general settings 
(cf. Lemma \ref{LemGromov1} and Lemma \ref{LemGromov2}).
	\subsection{Main results}
We want to understand a flexible sheaf as a fibrant object.
Furthermore, we hope to understand $\Delta : \mathcal{F} \rightarrow \mathcal{F}^\ast$ as a fibrant replacement of $\mathcal{F}$ by Lemma \ref{Gromov2}.
\begin{que}
Is there a model structure on ${\bf Sh}(B; {\bf qTop})$ such that the followings hold?
\begin{itemize}
\item A morphism $f$ in ${\bf Sh}(B; {\bf qTop})$ is a weak equivalence \\
$\Leftrightarrow$ $f$ is a stalkwise weak equivalence.
\item An object $\mathcal{F}$ in ${\bf Sh}(B; {\bf qTop})$ is fibrant \\
$\Leftrightarrow$ $\mathcal{F}$ is flexible.
\end{itemize}
\end{que}
In this paper, we give two answers.
The first result is that ${\bf Sh}(B; {\bf qTop})$ can be embedded in a model category. 
\setcounter{section}{3}
\setcounter{thm}{0}
\begin{thm}
There exists a right proper model category ${\bf PSh}_\ast(\tilde{\mathscr{U}}_B; {\bf qTop})$ and a fully faithful and left exact embedding $\iota : {\bf Sh}(B; {\bf qTop}) \rightarrow {\bf PSh}_\ast(\tilde{\mathscr{U}}_B; {\bf qTop})$ such that the followings hold. \par
Let $f : \mathcal{F} \rightarrow \mathcal{G}$ be a morphism in ${\bf Sh}(B; {\bf qTop})$.
\begin{itemize}
\item $\iota(f)$ is a weak equivalence \\
$\Leftrightarrow$ $f$ is a stalkwise weak equivalence
(i.e. $f_x : \mathcal{F}_x \rightarrow \mathcal{G}_x$ is a weak equivalence for any $x \in B$).
\item $\iota(f)$ is a fibration \\
$\Leftrightarrow$ $f$ is a flexible extension
(cf. Definition \ref{DefFlexible}).
\end{itemize}
In particular, for a continuous sheaf $\mathcal{F}$ on $B$, 
\begin{itemize}
\item $\iota(\mathcal{F})$ is fibrant \\
$\Leftrightarrow$ $\mathcal{F}$ is flexible
(cf. Definition \ref{DefFlexible}).
\end{itemize}
\end{thm}
Unfortunately, ${\bf Sh}(B; {\bf qTop})$ itself cannot be expected to have cofibratons.
Therefore, ${\bf Sh}(B; {\bf qTop})$ itself cannot be expected to have a model structure.
However, ${\bf Sh}(B; {\bf qTop})$ has fibrations and weak equivalences.
An {\it ABC (pre)fibration structure} is known as a categorical structure equipped with fibrations and weak equivalences
(cf. \cite{radulescu2006cofibrations} or Section 4).
The second result is that ${\bf Sh}(B; {\bf qTop})$ has an ABC prefibration structure.
\setcounter{section}{4}
\setcounter{thm}{0}
\begin{thm}
Let $B$ be a strongly locally contractible normal Hausdorff space.
There exists an ABC prefibration structure on ${\bf Sh}(B; {\bf qTop})$ such that the followings hold. \par
Let $f : \mathcal{F} \rightarrow \mathcal{G}$ be a morphism in ${\bf Sh}(B; {\bf qTop})$.
\begin{itemize}
\item $f$ is a weak equivalence \\
$\Leftrightarrow$ $f$ is a stalkwise weak equivalence
(i.e. $f_x : \mathcal{F}_x \rightarrow \mathcal{G}_x$ is a weak equivalence for any $x \in B$).
\item $f$ is a fibration \\
$\Leftrightarrow$ $f$ is a flexible extension
(cf. Definition \ref{DefFlexible}).
\end{itemize}
In particular, for a continuous sheaf $\mathcal{F}$ on $B$, 
\begin{itemize}
\item $\mathcal{F}$ is fibrant \\
$\Leftrightarrow$ $\mathcal{F}$ is flexible
(cf. Definition \ref{DefFlexible}).
\end{itemize}
\end{thm}
The proof of the first result is almost purely category-theoretic.
On the other hand, the proof of the second result reflects geometric considerations by Gromov.
Most of the axioms of the ABC prefibration category ${\bf Sh}(B; {\bf qTop})$ follow from Theorem \ref{Thm1}.
However, only the proof of the axiom of the existence of a decomposition needs a generalisation of Lemma \ref{Gromov2}
(cf. Lemma \ref{LemGromov1} and Lemma \ref{LemGromov2}).
\setcounter{section}{0}
\setcounter{thm}{3}
	\subsection{History of h-principles}
In the Cousin problem solution, Oka \cite{10.32917/hmj/1558490525} showed that the solvability of certain complex analytic problems can be attributed to the solvability of topological problems.
This is what is now called {\it Oka principle}, and it can be said that the germ of the philosophy of the h-principle already existed here. \par
For a complex manifold $X$, let $\mathcal{O}_X^\ast$ be the sheaf of non-zero holomorphic functions everywhere, and let $\mathcal{C}_X^\ast$ be the sheaf of non-zero continuous functions everywhere.
{\it Oka principle} can be stated in the following format.
\begin{thm}[Oka \cite{10.32917/hmj/1558490525}; cf. \cite{forstnerivc2011stein}]
Let $X$ be a Stein manifold.
Then, the map $H^1(X; \mathcal{O}_X^\ast) \rightarrow H^1(X; \mathcal{C}_X^\ast)$ induced by the inclusion $\mathcal{O}_X^\ast \subset \mathcal{C}_X^\ast$ is an isomorphism.
\end{thm}
The sheaf $\mathcal{O}_X^\ast$ (resp. $\mathcal{C}_X^\ast$) is regarded as the sheaf of holomorphic (resp. continuous) functions valued in $\mathbb{C}^\ast (\cong GL_1(\mathbb{C}))$.
The homology group $H^1(X; \mathcal{O}_X^\ast)$ (resp. $H^1(X; \mathcal{C}_X^\ast)$) is regarded as the set of isomorphism classes of holomorphic (resp. continuous) complex line bundles over $X$.
Then, the above theorem can be rephrased as follow.
\begin{cor}
Let $X$ be a Stein manifold.
For any $C^0$-almost complex line bundle, there exists only one holomorphic complex line bundle which is isomorphic to it, up to holomorphic isomorphic.
\end{cor}
Because $GL_1(\mathbb{C})$ is abelian, extensive skills in sheaf cohomologies can be used to prove {\it Oka principle}.
On the other hand, the general $GL_n(\mathbb{C})$ is not abelian.
Grauert showed that the same statement still holds true for $GL_n(\mathbb{C})$ in general.
\begin{thm}[Grauert \cite{Grauert1958}; cf. \cite{forstnerivc2011stein}] \label{Grauert}
Let $X$ be a Stein manifold.
For any $C^0$-almost complex vector bundle, there exists only one holomorphic complex vector bundle which is isomorphic to it, up to holomorphic isomorphic.
\end{thm}
I do not know whether {\it Oka principle} inspired Nash's work, but the idea of attributing the solvability of analytic problems to the solvability of topological problems reappears in the proof of {\it Nash embedding theorems} \cite{nash1954c1}.
This had a direct impact on the later work of Hirsch \cite{hirsch1959immersions} and Smale \cite{smale1959classification1, smale1959classification2}, and the prototype of the h-principle was completed as the h-principle for immersions of manifolds.
This is now known as {\it Hirsch-Smale h-principle}. \par
Let $M$ and $N$ be manifolds.
A {\it formal immersion} $M \rightarrow N$ is a pair $(f, \tilde{f})$ of a continuous map $f : M \rightarrow N$ and a continuous bundle map $\tilde{f} : TM \rightarrow TN$ over $f$ such that each $\tilde{f} : T_xM \rightarrow T_{f(x)}N$ is injective.
Let $\operatorname{Imm}(M, N)$ be the set of immersions $M \looparrowright N$.
Let $\operatorname{Imm^f}(M, N)$ be the set of formal immersions $M \rightarrow N$.
{\it Hirsch-Smale h-principle} is stated in the following format.
\begin{thm}[Hirsch \cite{hirsch1959immersions}, Smale \cite{smale1959classification1,smale1959classification2}]
Let $M$ and $N$ be manifolds.
If we have an inequality $\operatorname{dim}(M) < \operatorname{dim}(N)$, then the inclusion map $\operatorname{Imm}(M, N) \subset \operatorname{Imm^f}(M, N)$ is a weak homotopy equivalence.
\end{thm}
A similar theorem on submersions is given by Phillips.
Let $M$ and $N$ be manifolds.
A {\it formal submersion} $M \rightarrow N$ is a pair $(f, \tilde{f})$ of a continuous map $f : M \rightarrow N$ and a continuous bundle map $\tilde{f} : TM \rightarrow TN$ over $f$ such that each $\tilde{f} : T_xM \rightarrow T_{f(x)}N$ is surjective.
Let $\operatorname{Sub}(M, N)$ be the set of submersions $M \rightarrow N$.
Let $\operatorname{Sub^f}(M, N)$ be the set of formal submersions $M \rightarrow N$.
{\it Phillips h-principle} is stated in the following format.
\begin{thm}[Phillips \cite{phillips1967submersions}]
Let $M$ and $N$ be manifolds.
If $M$ is an open manifold, then the inclusion map $\operatorname{Sub}(M, N) \subset \operatorname{Sub^f}(M, N)$ is a weak homotopy equivalence.
\end{thm}
Let $B$ be a smooth manifold.
For each open subset $U, V (\subset B)$, we define
\[
\operatorname{Diff}_B(U, V) := \{ \mbox{A diffeomorphism } U \rightarrow V \}.
\]
$\operatorname{Diff}_B$ is a groupoid called a {\it pseudogroup}. \par
Let $p : E \rightarrow B$ be a smooth fibre bundle.
Let $J^rp : J^rE \rightarrow B$ be the $r$-jet bundle of $p$.
A {\it partial differential relation} is a subbundle of the $r$-jet bundle $J^rp$.
A partial differential relation $R$ is {\it open} if the subbundle $R$ is an open subset of $J^rp$.
For each open subset $U, V (\subset B)$, we define
\[
\operatorname{Diff}_{E/B}(U, V) := \{ \mbox{A diffeomorphism } \hat{\phi} : U \rightarrow V \, | \, p \circ \hat{\phi} = \phi \circ p \, (^\exists \phi : p(U) \rightarrow p(V))  \}.
\]
$\operatorname{Diff}_{E/B}$ is a groupoid.
The map $\hat{\phi} \mapsto \phi$ defines a groupoid morphism $\operatorname{Diff}_{E/B} \rightarrow \operatorname{Diff}_B$.
A fibre bundle $E \rightarrow B$ is {\it natural} if the groupoid morphism $\operatorname{Diff}_{E/B} \rightarrow \operatorname{Diff}_B$ has a section $\operatorname{Diff}_B \rightarrow \operatorname{Diff}_{E/B}$. \par
A partial differential relation $R$ on a natural fibre bundle $E \rightarrow B$ is {\it $\operatorname{Diff}_B$-invariant} if a section $\operatorname{Diff}_B \rightarrow \operatorname{Diff}_{E/B}$ as above preserves $R$.
{\it Gromov h-principle} is stated in the following format.
\begin{thm}[Gromov \cite{gromov1971topological}; cf. \cite{eliashberg2002introduction, gromov2013partial}]
Let $p : E \rightarrow B$ be a smooth natural fibre bundle.
Let $R$ be an open, $\operatorname{Diff}_B$-invariant partial differential relation on $p$.
If $M$ is an open manifold, $R$ satisfies the parametric h-principle.
\end{thm}
	\subsection{Advantages of the sheaf theoretic h-principle}
{\it Hirsch-Smale h-principle} holds even if the domain manifold $M$ is not open.
To show this, it is necessary to discuss the {\it microextension trick}.
One advantage of the {\it sheaf theoretic h-principle} is that we can state the argument for the {\it microextension trick} as the {\it microextension theorem}.
Another advantage is that it has the potential to be applied to geometries (e.g. metric spaces) other than smooth manifolds.
In this section, another advantage, which is different from the two above, is discussed in more detail.
The advantage is that it has the potential for more generalizations, such as including {\it Haefliger h-principle} \cite{AST_1984__116__70_0} and {\it Oka-Grauert principle} \cite{10.32917/hmj/1558490525, Grauert1958}.
\vspace{10pt}\\
Any integrable distribution can be realised locally as the kernel of a submersion.
On the basis of this observation, Phillips \cite{phillips1968foliations, phillips1969foliations} was thinking of applying {\it Phillips h-principle} to foliations.
Haefliger \cite{AST_1984__116__70_0} took this a step further and gave a classification of homotopy classes of integrable distributions. \par
Haefliger's idea is that any integrable distribution can be expressed as a kernel of submersion into a ``universal $q$-dimensional manifold" $\Gamma^q$.
Here, the object $\Gamma^q$ is the \'{e}tale groupoid corresponding to the pseudogroup $\operatorname{Diff}_{\mathbb{R}^q}$ consisting of local diffeomorphisms on $\mathbb{R}^q$.
{\it Haefliger h-principle} is stated in the following format.
\begin{thm}[Haefliger \cite{AST_1984__116__70_0}]
Let $M$ be a smooth manifold.
If $M$ is an open manifold, there exists a bijection between the following.
\begin{enumerate}
\item The integrable homotopy classes of integrable distribution on $M$ whose corank is $q$.
\item The integrable homotopy classes of the continuous bundle map $TM \rightarrow N\Gamma^q$ such that it is surjective on each fibre, where $N\Gamma^q$ is the classifying space of the tangent bundle $T\Gamma^q$ of $\Gamma^q$.
\end{enumerate}
\end{thm}
{\it Haefliger h-principle} means that ``$\pi_0(\operatorname{Sub}(M, \Gamma^q)) \rightarrow \pi_0(\operatorname{Sub^f}(M, \Gamma^q))$ is bijection".
\vspace{10pt}\\
Let $X$ and $Y$ be complex manifolds.
Let $\mathcal{O}(X, Y)$ be the set of holomorphic maps $X \rightarrow Y$.
Let $\mathcal{C}(X, Y)$ be the set of continuous maps $X \rightarrow Y$.
The following properties are important in the Oka-Grauert Theory.
\begin{thm}[\mbox{cf. \cite[Corollary 5.5.6]{forstnerivc2011stein}}]
Let $X$ be a Stein manifold, and let $Y$ be an Oka manifold
(e.g. complex Lie group).
Then, $\mathcal{O}(X, Y) \subset \mathcal{C}(X, Y)$ is a weak homotopy equivalence.
\end{thm}
A complex vector bundle on $X$ corresponds one-to-one to a $GL_n(\mathbb{C})$-structure on $X$.
In fact, a $G$-structure on $X$ is a {\it generalized map} $X \rightarrow G$, where $G$ is any Lie groupoid.
We consider a Lie group to be a Lie groupoid whose object space is a singleton.
{\it Oka-Grauert principle} (Theorem \ref{Grauert}) means that ``$\pi_0(\mathcal{O}(X, GL_n(\mathbb{C}))) \rightarrow \pi_0(\mathcal{C}(X, GL_n(\mathbb{C})))$ is bijection".
\vspace{10pt}\\
In order to unify {\it Haefliger h-principle} and {\it Oka-Grauert principle} into {\it Gromov h-principle}, we need to develop the {\it stack theoretic h-principle}.
To approach this dream, first, need to develop the {\it sheaf theoretic h-principle}.
\setcounter{section}{0}
\setcounter{thm}{0}
	\section{Quasitopological spaces}
Some different topologies can be admitted into the space of maps, for example, the {\it Whitney topology}.
However, if we admit the quotient topology into the stalk, the topologies will be collapsed.
Here we want to deal with the germs while keeping the local information.
The solution to this problem is a {\it quasitopological space}.
The discussion in Section 1.1.4 solves this problem. \par
The idea of a quasitopological space is credited to Spanier and Whithead \cite{spanier1962theory}.
Spanier \cite{spanier1963quasi} has written a paper on quasitopological spaces.
A quasitopological space dealt with in this paper is not exactly the same as the one they defined, but it is almost similar. \par
In this paper, a quasitopological space is a sheaf on a certain small site.
(For more information on sites and sheaves, see \cite{kashiwara18001categories}.)
The category of quasitopological spaces has a model structure.
	\subsection{Quasitopological spaces}
We will define a quasitopological space and check its properties in this subsection.
The discussion in Section 1.1.4 is often used when we have a stalkwise discussion on continuous sheaves.
The discussion in Section 1.1.5 implies that finite polyhedra (e.g. simplexes, horns, their boundaries, etc.) can be treated in the same way as topological spaces.
	\subsubsection{Prequasitopological spaces}
Let {\bf Sing$\Delta$} be a small category defined as follows.
\begin{itemize}
\item An object of {\bf Sing$\Delta$} is a closed subset of a simplex $\Delta^n$.
\item A morphism of {\bf Sing$\Delta$} is a continuous map.
\end{itemize}
\begin{defi}
A {\it prequasitopological space} is a presheaf on the small category {\bf Sing$\Delta$}.
A morphism between prequasitopological spaces is a morphism between presheaves.
{\bf pqTop} is the category of prequasitopological spaces.
\end{defi}
The category {\bf pqTop} is the category ${\bf PSh}({\bf Sing\Delta})$ of presheaves.
This category is complete and cocomplete and Cartesian closed.
For a prequasitopological space $X$ and a closed subset $S$ of a simplex, denote by $X[S]$ the set to which $S$ is mapped by a functor $X : {\bf Sing\Delta}^{op} \rightarrow {\bf Set}$, where {\bf Set} is the category of sets and maps.
\begin{ex}
For a topological space $X$, we define a prequasitopological space $\operatorname{Sing^q}(X)$ as the following:
\[
\operatorname{Sing^q}(X)[S] := \{ \mbox{continuous map } S \rightarrow X \}.
\]
The construction $X \mapsto \operatorname{Sing^q}(X)$ defines a functor $\operatorname{Sing^q} : {\bf Top} \rightarrow {\bf pqTop}$.
In particular, if $X$ is a closed subset of a simplex, we obtain an equality $\operatorname{Sing^q}(X) = {\bf Sing\Delta}(-, X)$.
If $X = \Delta^n$, and there is no confusion, then $\operatorname{Sing^q}(\Delta^n)$ is denoted as $\Delta^n$.
\end{ex}
\begin{rem}
Each element $s \in X[S]$ corresponds one-to-one with a morphism $s' : \operatorname{Sing^q}(S) \rightarrow X$ by {\it Yoneda's Lemma}.
We often equate $s$ with $s'$, and use the same symbol.
\end{rem}
	\subsubsection{Quasitopological spaces}
The small category {\bf Sing$\Delta$} is a small site by defining a covering of an object $S$ as a finite closed covering of the topological space $S$.
\begin{defi} \label{DefQTop}
A {\it quasitopological space} is a sheaf on the small site {\bf Sing$\Delta$}.
A morphism between quasitopological spaces is a morphism between sheaves.
{\bf qTop} is the category of quasitopological spaces.
\end{defi}
The category {\bf qTop} is the category ${\bf Sh}({\bf Sing\Delta})$ of sheaves.
This category is complete and cocomplete and Cartesian closed.
\begin{ex}
For a topological space $X$, the prequasitopological space $\operatorname{Sing^q}(X)$ is a quasitopological space.
\end{ex}
	\subsubsection{Geometric realizations}
We define a functor $G_2 : {\bf Top} \rightarrow {\bf qTop}$ as $G_2(X) := \operatorname{Sing^q}(X)$.
The functor $G_2$ is the left Kan extension of Yoneda embedding $y : {\bf Sing\Delta} \rightarrow {\bf qTop}$ along the inclusion functor ${\bf Sing\Delta} \hookrightarrow {\bf Top}$.
	\[\xymatrix{
		{\bf qTop}
	\\
		{\bf Sing\Delta} \ar[u]^-{y} \ar@{^(->}[rr]
		&& {\bf Top} \ar[ull]_-{G_2}.
	}\]
Let $F_2 : {\bf qTop} \rightarrow {\bf Top}$ be the left Kan extension of the inclusion functor ${\bf Sing\Delta} \hookrightarrow {\bf Top}$ along $y$.
It is classically known that there exists an adjunction $F_2 \dashv G_2$.
The functor $F_2$ is called the {\it geometric realizations}.
Specifically, $F_2$ is defined as $F_2 (X) := \int^S {\bf qTop}(S, X) \bullet S$
(cf. \cite[Chapter 10, Section 4]{mac2013categories}),
where the integral symbol $\int$ is the coend and ``$(-) \bullet (-)$" is the copower object.
	\subsubsection{Filtered colimit trick}
A limit on the category {\bf qTop} can be calculated objectwise.
(i.e. For any small diagram $X_\lambda$ in {\bf qTop} and any object $S$ in {\bf Sing$\Delta$}, we have an isomorphism $(\displaystyle\lim_{\longleftarrow} X_\lambda)[S] \cong \displaystyle\lim_{\longleftarrow} (X_\lambda [S])$.)
A colimit cannot be calculated in the same way in general.
However, because the covering on the site {\bf Sing$\Delta$} is finite, a filtered colimit can be calculated objectwise.
\begin{prop} \label{PropFilter}
For any closed subset $S$ of a simplex, the evaluation functor $\operatorname{ev}_S : {\bf qTop} \rightarrow {\bf Set} ; X \mapsto X[S]$ preserves any filtered colimit.
\end{prop}
\begin{proof}
\setcounter{equation}{0}
Take any closed subset $S$ of a simplex and any filtered diagram $X_\lambda$ ($\lambda \in \Lambda$) in {\bf qTop}.
We will show that the following canonical map is an isomorphism:
\begin{equation}
\displaystyle\lim_{\substack{\longrightarrow \\ \lambda}} (X_\lambda [S]) \rightarrow (\displaystyle\lim_{\substack{\longrightarrow \\ \lambda}} X_\lambda) [S].
\end{equation}
For any index $\lambda \in \Lambda$ and any finite closed covering $\{ S_i \}_i$ of $S$, the following diagram is an equalizer:
\[
X_\lambda [S] \rightarrow \prod_i X_\lambda [S_i] \rightrightarrows \prod_{i, j} X_\lambda [S_i \cap S_j].
\]
Filtered colimits commute with finite limits in {\bf Set}
(cf. \cite[Chapter 9, Section 2]{mac2013categories}).
Then, the following diagram is an equalizer:
\begin{equation}
\displaystyle\lim_{\substack{\longrightarrow \\ \lambda}} (X_\lambda [S]) \rightarrow \prod_i \displaystyle\lim_{\substack{\longrightarrow \\ \lambda}} (X_\lambda [S_i]) \rightrightarrows \prod_{i, j} \displaystyle\lim_{\substack{\longrightarrow \\ \lambda}} (X_\lambda [S_i \cap S_j]).
\end{equation}
We define a prequasitopological space $Y$ as the following:
\[
Y[S] := \displaystyle\lim_{\substack{\longrightarrow \\ \lambda}} (X_\lambda [S]).
\]
$Y$ is a quasitopological space because the diagram (2) is an equalizer.
The colimit of the diagram $X_\lambda$ is the sheafification of $Y$.
However, $Y$ is a sheaf, therefore the colimit of the diagram $X_\lambda$ is $Y$.
This means that the canonical map (1) is an isomorphism.
Then, for any closed subset $S$ of a simplex, the evaluation functor $\operatorname{ev}_S$ preserves any filtered colimit.
\end{proof}
The Hom-functor ${\bf qTop}(\operatorname{Sing^q}(S), -) : {\bf qTop} \rightarrow {\bf Set}$ is naturally isomorphic to the evaluation functor $\operatorname{ev}_S$ by {\it Yoneda's Lemma}.
Proposition \ref{PropFilter} means that, for any filtered diagram $X_\lambda$ ($\lambda \in \Lambda$) in {\bf qTop}, we have the following isomorphism:
\[
{\bf qTop}(\operatorname{Sing^q}(S), \displaystyle\lim_{\substack{\longrightarrow \\ \lambda}} X_\lambda) \cong \displaystyle\lim_{\substack{\longrightarrow \\ \lambda}} {\bf qTop}(\operatorname{Sing^q}(S), X_\lambda).
\]
We obtain the following proposition.
\begin{prop} \label{PropFilterLift}
Let $S$ be a closed subset of a simplex, and let $X_\lambda$ ($\lambda \in \Lambda$) be a filtered diagram in {\bf qTop}.
For any morphism $f : \operatorname{Sing^q}(S) \rightarrow \displaystyle\lim_{\longrightarrow} X_\lambda$, we have the followings.
\begin{itemize}
\item There exist an index $\lambda$ and a morphism $\tilde{f} : \operatorname{Sing^q}(S) \rightarrow X_\lambda$ such that the following diagram is commutative:
	\[\xymatrix{
		& X_\lambda \ar[d]
	\\
		\operatorname{Sing^q}(S) \ar[r]^-{f} \ar[ur]^-{\tilde{f}}
		& \displaystyle\lim_{\substack{\longrightarrow \\ \lambda}} X_\lambda.
	}\]
\item If there exist an index $\lambda'$ and a morphism $\tilde{f'}$, similar to the index $\lambda$ and the morphism $\tilde{f}$ above, then there exist two arrows $\lambda \rightarrow \mu \leftarrow \lambda'$ between indexes such that the following diagram is commutative:
	\[\xymatrix{
		\operatorname{Sing^q}(S) \ar[r]^-{\tilde{f'}} \ar[d]_-{\tilde{f}}
		& X_{\lambda'} \ar[d]
	\\
		X_\lambda \ar[r]
		& X_\mu.
	}\]
\end{itemize}
\end{prop}
\begin{proof}
The statement is a rephrasing of the definition of a filtered colimit in {\bf Set}.
\end{proof}
We consider the following diagram:
	\[\xymatrix{
		\operatorname{Sing^q}(S) \ar[r] \ar[d]
		& \displaystyle\lim_{\substack{\longrightarrow \\ \lambda}} X_\lambda \ar[d]
	\\
		\operatorname{Sing^q}(T) \ar[r]
		& \displaystyle\lim_{\substack{\longrightarrow \\ \lambda}} Y_\lambda,
	}\]
where $S$ and $T$ are closed subsets of simplexes, and $X_\lambda$ and $Y_\lambda$ ($\lambda \in \Lambda$) are filtered diagrams in {\bf qTop}.
Then, there exists a following commutative diagram:
	\[\xymatrix{
		\operatorname{Sing^q}(S) \ar[r] \ar[d]
		& X_\lambda \ar[d]
	\\
		\operatorname{Sing^q}(T) \ar[r]
		& Y_\lambda.
	}\]
By the discussion in the later section, Section 1.1.5, we may accept $S$ and $T$ as finite polyhedra.
	\subsubsection{Patching Lemma}
The adjunction $F_2 \dashv G_2 : {\bf qTop} \rightarrow {\bf Top}$ is not an equivalence.
However, if $X$ is a finite polyhedron in {\bf Top} (resp. {\bf qTop}), ${\bf Top}(X, Y) \rightarrow {\bf qTop}(\operatorname{Sing^q}(X), \operatorname{Sing^q}(Y))$ (resp. ${\bf qTop}(X, Y) \rightarrow {\bf Top}(F_2(X), F_2(Y))$) is a bijection.
As such, if $X$ is a finite polyhedron, then the quasitopological space $\operatorname{Sing^q}(X)$ behaves in the almost same way as the topological space $X$.
These can be seen from the following proposition.
\begin{lem} \label{LemPatching}
Let $X$ be a topological space, and let $\{ X_i \}_i$ be a finite closed covering of $X$.
Then the following diagram is a coequalizer.
\[
\coprod_{i, j} \operatorname{Sing^q}(X_i) \cap \operatorname{Sing^q}(X_j) \rightrightarrows \coprod_i \operatorname{Sing^q}(X_i) \rightarrow \operatorname{Sing^q}(X)
\]
\end{lem}
\begin{proof}
\setcounter{equation}{0}
We will show that there exists $f'$ in the following diagram.
\begin{equation}
\vcenter{\xymatrix{
		\displaystyle\coprod_{i, j} \operatorname{Sing^q}(X_i) \cap \operatorname{Sing^q}(X_j) \ar@<2pt>[r] \ar@<-2pt>[r]
		& \displaystyle\coprod_i \operatorname{Sing^q}(X_i) \ar[r] \ar[dr]_-{f}
		& \operatorname{Sing^q}(X) \ar@{.>}[d]^-{f'}
	\\
		&& Y
	}}
\end{equation}
Take any $s \in \operatorname{Sing^q}(X)[S] (= {\bf Top}(S, X))$.
Let $S_i := s^{-1}(X_i)$.
Then $\{ S_i \}_i$ is a finite closed covering of $S$, and $s |_{S_i} \in \operatorname{Sing^q}(X_i)[S_i] (= {\bf Top}(S_i, X_i))$.
Let $t_i := f(s |_{S_i}) \in Y[S_i]$.
Then, we have the following equality:
\[
\begin{array}{rcl}
t_i |_{S_i \cap S_j}
& = & f(s |_{S_i}) |_{S_i \cap S_j} \\
& = & f(s |_{S_i \cap S_j}) \\
& = & f(s |_{S_j}) |_{S_i \cap S_j} \\
& = & t_j |_{S_i \cap S_j}.
\end{array}
\]
Then, the family of elements $(t_i)_i \in \prod Y[S_i]$ lifts a unique element $t \in Y[S]$.
If we define $f'(s) := t$, the above diagram (1) is commutative.
Conversely, if there exists $f'$ such that the above diagram (1) is commutative, this coinside with the above definition by uniqueness of $t$.
\end{proof}
\begin{prop} \label{PropFiniteFullyFaithful}
\mbox{}
\begin{enumerate}
\item If $X$ is a finite polyhedron in {\bf Top}, ${\bf Top}(X, Y) \rightarrow {\bf qTop}(\operatorname{Sing^q}(X), \operatorname{Sing^q}(Y))$ is a bijection.
\item If $X$ is a finite polyhedron in {\bf qTop}, ${\bf qTop}(X, Y) \rightarrow {\bf Top}(F_2(X), F_2(Y))$ is a bijection.
\end{enumerate}
\end{prop}
\begin{proof}
We will show the claim {\it 1.}.
The proof of the claim {\it 2.} is similar. \par
If $X$ is a finite polyhedron, there exists a finite closed covering $\{ X_i \}_i$ of $X$ such that each $X_i$ is homeomorphic to a simplex.
Let $X'$, $Y'$ and $X'_i$ be the quasitopological spaces $\operatorname{Sing^q}(X)$, $\operatorname{Sing^q}(Y)$ and $\operatorname{Sing^q}(X_i)$ respectively.
Then, ${\bf Top}(X_i, Y) \rightarrow {\bf qTop}(X'_i, Y')$ and ${\bf Top}(X_i \cap X_j, Y) \rightarrow {\bf qTop}(X'_i \cap X'_j, Y')$ are bijections.
We consider the following diagram:
	\[\xymatrix{
		{\bf Top}(X, Y) \ar[r] \ar[d]
		& \displaystyle\prod_i {\bf Top}(X_i, Y) \ar@<2pt>[r] \ar@<-2pt>[r] \ar[d]
		& \displaystyle\prod_{i, j} {\bf Top}(X_i \cap X_j, Y) \ar[d]
	\\
		{\bf qTop}(X', Y') \ar[r]
		& \displaystyle\prod_i {\bf qTop}(X'_i, Y') \ar@<2pt>[r] \ar@<-2pt>[r]
		& \displaystyle\prod_{i, j} {\bf qTop}(X'_i \cap X'_j, Y').
	}\]
The upper row is an equalizer.
The bottom row is an equalizer by Lemma \ref{LemPatching}.
The right two columns are isomorphisms.
Then, the left column is an isomorphism.
\end{proof}
${\bf Top^f}$ is the category of finite polyhedra in {\bf Top}. \par
${\bf qTop^f}$ is the category of finite polyhedra in {\bf qTop}.
\begin{cor} \label{CorFiniteEquivalence}
The adjunction $F_2 \dashv G_2$ indeces a categorical equivalence ${\bf qTop^f} \simeq {\bf Top^f}$.
\end{cor}
\begin{proof}
It is obvious by Proposition \ref{PropFiniteFullyFaithful}.
\end{proof}
	\subsection{Some model structures}
In this subsection, we will show that the category {\bf qTop} of quasitopological spaces has a model structure.
Furthermore, we will review Homotopy Theory in {\bf qTop}.
\vspace{10pt}\\
Consider the following adjunctions sequence:
	\[\xymatrix{
		{\bf sSet} \ar@<-10pt>[r]_{F_0} \ar@{}[r] | {\top}
		& {\bf pqTop} \ar@<-10pt>[l]_{G_0} \ar@<-10pt>[r]_{F_1} \ar@{}[r] | {\top}
		& {\bf qTop} \ar@<-10pt>[l]_{G_1} \ar@<-10pt>[r]_{F_2} \ar@{}[r] | {\top}
		& {\bf Top} \ar@<-10pt>[l]_{G_2}.
	}\]
The above categories and functors are defined as follows.
\begin{itemize}
\item {\bf sSet} is the category of simplicial sets.
\item {\bf pqTop} is the category of prequasitopological spaces (cf. Section 1.1.1).
\item {\bf qTop} is the category of quasitopological spaces (cf. Section 1.1.2).
\item {\bf Top} is the category of topological spaces and continuous maps.
\item $G_0$ is the functor ${\bf Set}^{{\bf Sing\Delta}^{op}} \rightarrow {\bf Set}^{{\bf \Delta}^{op}}$ induced by the inclusion functor ${\bf \Delta} \hookrightarrow {\bf Sing\Delta}$, where ${\bf \Delta}$ is the simplicial category.
\item $G_1$ is the forgetful functor.
\item $G_2$ is defined as $G_2(X) := \operatorname{Sing^q}(X)$ (cf. Section 1.1.3).
\item $F_1$ is the sheafification.
\item $F_0$ and $F_2$ are defined in the same way as the geometric realizations.
Specifically, these are defined as follows.
	\begin{itemize}
	\item $F_0(X) := \int^n {\bf sSet}(\Delta^{n-1}, X) \bullet \Delta^{n-1}$.
	\item $F_2(X) := \int^S {\bf qTop}(S, X) \bullet S$ (cf. Section 1.1.3).
	\end{itemize}
Each is the left Kan extension of the appropriate diagram.
\end{itemize}\par
We will show that {\bf pqTop} (resp. {\bf qTop}) has a model structure by {\it Transfer Theorem}
(cf. Theorem \ref{ThmTransfer}).
Furthermore, all the above adjunctions are Quillen equivalences.
	\subsubsection{A model structure for prequasitopological spaces}
We will show that the category {\bf pqTop} has a model structure in this subsubsection.
We will prepare for this.
\vspace{10pt}\\
Let $I_{\bf sSet}$, $J_{\bf sSet}$ and $W_{\bf sSet}$ be the sets of some morphisms in {\bf sSet} defined as the followings:
\[
\begin{array}{rcl}
I_{\bf sSet}
& := & \{ \partial \Delta^n \hookrightarrow \Delta^n \, | \, n \}, \\
J_{\bf sSet}
& := & \{ \Lambda^{n-1}_i \hookrightarrow \Delta^n \, | \, n, i \}.
\end{array}
\]
It is classically known that the category {\bf sSet} has a cofibrantly generated model structure such that the class of cofibrations is generated by $I_{\bf sSet}$ and the class of acyclic cofibrations is generated by $J_{\bf sSet}$
(cf. Definition \ref{DefCofibrantlyGenerated}).
Let $W_{\bf sSet}$ be the class of weak equivalences in {\bf sSet}.
\begin{lem} \label{LemMonoPreserve}
The functor $F_0 : {\bf sSet} \rightarrow {\bf pqTop}$ preserves any monomorphism.
\end{lem}
\begin{proof}
Any monomorphism between simplicial sets is a cofibration.
$F_0$ preserves any push-out, any transfinite composition and any retract.
Then, we obtain the inclusion relation $F_0(\operatorname{Cof}(I_{\bf sSet})) \subset \operatorname{Cof}(F_0I_{\bf sSet})$.
Any morphism $\partial \Delta^n \rightarrow \Delta^n$ in the set $F_0I_{\bf sSet}$ is a monomorphism in {\bf pqTop}.
Any push-out, any transfinite composition and any retract preserve monomorphisms in {\bf pqTop} by the objectwise calculation.
Then, any morphism belonging to $\operatorname{Cof}(F_0I_{\bf sSet})$ is a monomorphism.
Therefore, the functor $F_0$ preserves any monomorphism.
\end{proof}
\begin{lem} \label{LemPrequasiCounit}
Let $\epsilon : F_0G_0 \rightarrow \operatorname{Id}_{\bf pqTop}$ be the counit of the adjunction $F_0 \dashv G_0 : {\bf sSet} \rightarrow {\bf pqTop}$.
For any prequasitopological space $X$, the morphism $\epsilon_X$ has the right lifting property with respect to $F_0I_{\bf sSet}$.
\end{lem}
\begin{proof}
For any square in the following diagram, we will show that there is a morphism $\gamma$ below.
	\[\xymatrix{
		\partial \Delta^n \ar[r]^-{\alpha} \ar@{^(->}[d]
		& F_0(G_0(X)) \ar[d]^-{\epsilon_X}
	\\
		\Delta^n \ar[r]_-{\beta} \ar@{.>}[ur]^-{\gamma}
		& X.
}\]
By definition, we obtain the following equality:
\[
\begin{array}{rcl}
F_0(G_0(X))
& = & \int^k {\bf sSet}(\Delta^{k-1}, G_0(X)) \bullet \Delta^{k-1} \\
& = & \int^k G_0(X)[\Delta^{k-1}] \bullet \Delta^{k-1} \\
& = & \int^k X[\Delta^{k-1}] \bullet \Delta^{k-1}.
\end{array}
\]
Then, we have the quotient morphism $\coprod_k X[\Delta^k] \bullet \Delta^k \rightarrow F_0(G_0(X))$.
Let $\gamma$ be the composition of the sequence $\Delta^n \overset{\beta}{\rightarrow} X[\Delta^n] \bullet \Delta^n \rightarrow \coprod_k X[\Delta^k] \bullet \Delta^k \rightarrow F_0(G_0(X))$.
Obviously, we have the equality $\epsilon_X \circ \gamma = \beta$.
For any embedding $\sigma_i : \Delta^{n-1} \hookrightarrow \Delta^n$, the following diagram is commutative:
	\[\xymatrix{
		\Delta^{n-1} \ar[r]^-{\alpha \circ \sigma_i} \ar@{^(->}[d]_-{\sigma_i}
		& F_0(G_0(X)) \ar[d]^-{\epsilon_X}
	\\
		\Delta^n \ar[r]_-{\beta}
		& X.
}\]
Then, the following diagram is commutative:
	\[\xymatrix{
		\Delta^{n-1} \ar[r]^-{\alpha \circ \sigma_i} \ar@{^(->}[d]_-{\sigma_i}
		& X[\Delta^{n-1}] \bullet \Delta^{n-1} \ar[r]^-{id \bullet \sigma_i}
		& X[\Delta^{n-1}] \bullet \Delta^n
	\\
		\Delta^n \ar[r]_-{\beta}
		& X[\Delta^n] \bullet \Delta^n \ar[ur]_-{X[\sigma_i] \bullet id}.
}\]
This means that the following diagram is commutative:
	\[\xymatrix{
		\Delta^{n-1} \ar@{^(->}[r]^-{\sigma_i} \ar@{^(->}[d]_-{\sigma_i}
		& \partial \Delta^n \ar[d]^-{\alpha}
	\\
		\Delta^n \ar[r]_-{\gamma}
		& F_0(G_0(X)).
}\]
We obtain the equality $\gamma |_{\partial \Delta^n} = \alpha$.
Then, the morphism $\epsilon_X$ has the right lifting property with respect to $F_0I_{\bf sSet}$.
\end{proof}
\begin{lem} \label{LemPrequasiUnit}
Let $\eta : \operatorname{Id}_{\bf sSet} \rightarrow G_0F_0$ be the unit of the adjunction $F_0 \dashv G_0 : {\bf sSet} \rightarrow {\bf pqTop}$.
\begin{description}
\item[(1)] $\eta_{\Delta^n} : \Delta^n \rightarrow G_0(F_0(\Delta^n))$ is a weak equivalence.
\item[(2)] $\eta_{\Lambda^n_i} : \Lambda^n_i \rightarrow G_0(F_0(\Lambda^n_i))$ is a weak equivalence.
\end{description}
\end{lem}
\begin{proof}
The claim {\bf (1)} is obvious by the equality $G_0(F_0(\Delta^n)) = \operatorname{Sing}(\Delta^n)$. \par
We will show the claim  {\bf (2)} by induction on $n$.
If $n = 0$, the morphism $\eta_{\Lambda^0_0} : {\bf 1} \rightarrow {\bf 1}$ is an isomorphism.
This is a weak equivalence. \par
Assume that for any $k$ such that $k < n$, the claim that $n$ is replaced by $k$ holds.
Let $\Gamma^k$ be the simplicial set generated by $k$-surfaces of $\Lambda^n_i$ containing the $i$-th vertex.
We consider the following diagram:
	\[\xymatrix{
		\Gamma^0 \ar@{^(->}[r] \ar@{=}[d]
		& \Gamma^1 \ar@{^(->}[r] \ar@{^(->}[d]
		& \cdots \ar@{^(->}[r]
		& \Gamma^k \ar@{^(->}[r] \ar@{^(->}[d]
		& \cdots \ar@{^(->}[r]
		& \Gamma^{n-1} = \Lambda^n_i \ar@{^(->}[d]^-{\eta_{\Lambda^n_i}}
	\\
		G_0(F_0(\Gamma^0) \ar@{^(->}[r]
		& G_0(F_0(\Gamma^1) \ar@{^(->}[r]
		& \cdots \ar@{^(->}[r]
		& G_0(F_0(\Gamma^k) \ar@{^(->}[r]
		& \cdots \ar@{^(->}[r]
		& G_0(F_0(\Gamma^{n-1}) = G_0(F_0(\Lambda^n_i).
}\]
We will show that $\Gamma^k \rightarrow G_0(F_0(\Gamma^k)$ is a weak equivalence by induction on $k$. \par
If $k = 0$, this is obvious because of the equality $\Gamma^0 = {\bf 1} = G_0(F_0(\Gamma^0)$. \par
Assume that for any $l$ such that $l < k$, $\Gamma^l \rightarrow G_0(F_0(\Gamma^l)$ is a weak equivalence.
$\Gamma^{k-1} \rightarrow G_0(F_0(\Gamma^{k-1})$ is a weak equivalence.
$\Gamma^{k-1} \rightarrow \Gamma^k$ is a finite relative $\{ \Lambda^k_j \rightarrow \Delta^k \, | \, j \}$-cell complex.
i.e. There exists a finite sequence:
\[
\Gamma^{k-1} = X_0 \rightarrow X_1 \rightarrow \cdots \rightarrow X_N = \Gamma^k,
\]
where each component $X_{s-1} \rightarrow X_s$ is a push-out of $\Lambda^k_j \rightarrow \Delta^k$.
$F_0$ preserves any colimit because it is a left adjoint functor.
$G_0$ preserves any colimit by the objectwise calculation.
$G_0(F_0(\Gamma^{k-1}) \rightarrow G_0(F_0(\Gamma^k)$ is a finite relative $\{ G_0(F_0(\Lambda^k_j) \rightarrow G_0(F_0(\Delta^k) \, | \, j \}$-cell complex.
If we show that $G_0(F_0(\Lambda^k_j) \rightarrow G_0(F_0(\Delta^k)$ is an acyclic cofibration, We obtain that $G_0(F_0(\Gamma^{k-1}) \rightarrow G_0(F_0(\Gamma^k)$ is an acyclic cofibration.
We will show that $G_0(F_0(\Lambda^k_j) \rightarrow G_0(F_0(\Delta^k)$ is an acyclic cofibration. \par
$\Lambda^k_j \rightarrow G_0(F_0(\Lambda^k_j))$ is a weak equivalence by the assumption of induction.
$G_0(F_0(\Lambda^k_j) \rightarrow G_0(F_0(\Delta^k)$ is a weak equivalence by the following diagram:
	\[\xymatrix{
		\Lambda^k_j \ar[r]^-{\simeq} \ar[d]_-{\simeq}
		& \Delta^k \ar[d]^-{\simeq}
	\\
		G_0(F_0(\Lambda^k_j)) \ar[r]
		& G_0(F_0(\Delta^k)).
}\]
$F_0$ preserves any monomorphism by Lemma \ref{LemMonoPreserve}.
$G_0$ preserves any monomorphism by the objectwise calculation.
$G_0(F_0(\Lambda^k_j) \rightarrow G_0(F_0(\Delta^k)$ is a cofibration.
Then, this is an acyclic cofibration. \par
Therefore, any relative $\{ G_0(F_0(\Lambda^k_j) \rightarrow G_0(F_0(\Delta^k) \, | \, j \}$-cell complex is an acyclic cofibration.
Then, $G_0(F_0(\Gamma^{k-1}) \rightarrow G_0(F_0(\Gamma^k)$ is a weak equivalence because it is an acyclic cofibration.
$\Gamma^k \rightarrow G_0(F_0(\Gamma^k)$ is a weak equivalence by the following diagram:
	\[\xymatrix{
		\Gamma^{k-1} \ar[r]^-{\simeq} \ar[d]_-{\simeq}
		& \Gamma^k \ar[d]
	\\
		G_0(F_0(\Gamma^{k-1})) \ar[r]_-{\simeq}
		& G_0(F_0(\Gamma^k)).
}\]
This completes the induction on $k$.
Therefore, for any $k (\le n)$, $\Gamma^k \rightarrow G_0(F_0(\Gamma^k)$ is a weak equivalence.
In particular, by considering the case $k = n$, $\eta_{\Lambda^n_i} : \Lambda^n_i \rightarrow G_0(F_0(\Lambda^n_i))$ is a weak equivalence.
This completes the induction on $n$.
Therefore, for any $n$, $\eta_{\Lambda^n_i} : \Lambda^n_i \rightarrow G_0(F_0(\Lambda^n_i))$ is a weak equivalence.
\end{proof}
Let $I_{\bf pqTop}$, $J_{\bf pqTop}$ and $W_{\bf pqTop}$ be sets as the followings:
\[
\begin{array}{rcl}
I_{\bf pqTop}
& := & F_0I_{\bf sSet}, \\
J_{\bf pqTop}
& := & F_0J_{\bf sSet}, \\
W_{\bf pqTop}
& := & G_0^{-1}W_{\bf sSet},
\end{array}
\]
Now we are ready to prove that {\bf pqTop} has a model structure.
\begin{thm} \label{ThmModePqtop}
{\bf pqTop} has a cofibrantly generated model structure satisfying the followings:
\begin{itemize}
\item $I_{\bf pqTop}$ generates cofibrations,
\item $J_{\bf pqTop}$ generates acyclic cofibrations, and
\item $W_{\bf pqTop}$ is the class of weak equivalences.
\end{itemize}
Furthermore, $F_0 \dashv G_0 : {\bf sSet} \rightarrow {\bf pqTop}$ is a Quillen equivalence.
\end{thm}
\begin{proof}
In order to apply Theorem \ref{ThmTransfer}, we show the followings:
\begin{description}
\item[(1)] $I_{\bf pqTop}$ and $J_{\bf pqTop}$ permit the small object argument.
\item[(2)]  We have the inclusion relation $\operatorname{cell}(J_{\bf pqTop}) \subset W_{\bf pqTop}$.
\end{description}\par
We will show the condition {\bf (1)}.
Any simplex $\Delta^n$ is small in {\bf pqTop} by Example \ref{ExSmallPresheaf}.
$\partial \Delta^n$ and $\Lambda^n_i$ are finite polyhedra.
These are small in {\bf pqTop} by Proposition \ref{PropSmallColimit}. \par
We will show the condition {\bf (2)}.
$F_0$ preserves any monomorphism by Lemma \ref{LemMonoPreserve}.
$G_0$ preserves any monomorphism by the objectwise calculation.
Then, $G_0(F_0(\Lambda^n_i \hookrightarrow \Delta^n))$ is a monomorphism.
This means that it is a cofibration in {\bf sSet}.
Furthermore, it is a weak equivalence by Lemma \ref{LemPrequasiUnit}.
We obtain the inclusion relation $G_0(J_{\bf pqTop}) \subset \operatorname{Cof}_{\bf sSet} \cap W_{\bf sSet}$ because $G_0(F_0(\Lambda^n_i \hookrightarrow \Delta^n))$ is an acyclic cofibration.
$G_0$ preserves any colimit by the objectwise calculation.
We obtain the sequence of the inclusion relations $G_0(\operatorname{cell}(J_{\bf pqTop})) = \operatorname{cell}(G_0(J_{\bf pqTop})) \subset \operatorname{Cof}_{\bf sSet} \cap W_{\bf sSet} \subset W_{\bf sSet}$.
Then, we obtain the inclusion relation $\operatorname{cell}(J_{\bf pqTop}) \subset G_0^{-1}(W_{\bf sSet}) = W_{\bf pqTop}$. \par
Therefpre, {\bf pqTop} has a cofibrantly generated model structure by Theorem \ref{ThmTransfer}. \par
We will show that the Quillen adjunction $F_0 \dashv G_0$ is a Quillen equivalence.
We consider the following diagram:
	\[\xymatrix{
		& {\bf pqTop} \ar[dr]^-{F_2 F_1 \dashv G_1 G_2}
	\\
		{\bf sSet} \ar[ur]^{F_0 \dashv G_0} \ar[rr]_{|-| \dashv \operatorname{Sing}}^{\simeq}
		&& {\bf Top},
	}\]
where $|-|$ is the {\it geometric realization} and $\operatorname{Sing}(-)$ is the {\it singular simplicial complex}.
It is classically known that $|-| \dashv \operatorname{Sing}$ is a Quillen equivalence.
Then, the Quillen adjunction $F_0 \dashv G_0$ is a ``split monomorphism".
Strictly speaking, the unit of the derived adjunction $\mathbb{L}F_0 \dashv \mathbb{R}G_0$ is a natural isomorphism.
The counit of the adjunction $F_0 \dashv G_0$ is a weak equivalence by Lemma \ref{LemPrequasiCounit}.
All objects in {\bf sSet} are cofibrant.
Then, the counit of the derived adjunction $\mathbb{L}F_0 \dashv \mathbb{R}G_0$ is a natural isomorphism.
$\mathbb{L}F_0 \dashv \mathbb{R}G_0$ is a categorical equivalence.
Therefore, $F_0 \dashv G_0$ is a Quillen equivalence.
\end{proof}
\begin{rem}
In fact, any acyclic cofibratoin in {\bf sSet} is a relative $J_{\bf sSet}$-cell complex.
However, I think, this is not well known.
Using this fact, the proof of Theorem \ref{ThmModePqtop} can be simplified slightly by Corollary \ref{CorTransfer}, which we will not explain here.
\end{rem}
	\subsubsection{A model structure for quasitopological spaces}
We will show that the category {\bf qTop} has a model structure in this subsubsection.
We will prepare for this.
\begin{lem} \label{LemSheafificationWeq}
Let $\eta$ be the unit of adjunction $F_1 \dashv G_1$.
For any prequasitopological space $X$, $\eta_X : X \rightarrow G_1(F_1(X))$ is a weak equivalence in {\bf pqTop}.
\end{lem}
\begin{proof}
$F_0 \dashv G_0$ and $|-| \dashv \operatorname{Sing}$ are Quillen equivalences.
Then, $F_2 F_1 \dashv G_1 G_2$ is a Quillen equivalence.
All objects in {\bf Top} are fibrant.
Then, the unit map of the adjunction $F_2 F_1 \dashv G_1 G_2$ is a weak equivalence.
We have an isomorphism $F_1(G_1(F_1(X))) \cong F_1(X)$ because the counit of the adjunction $F_1 \dashv G_1$ is a natural isomorphism.
We consider the following diagram:
	\[\xymatrix{
		X \ar[d]_-{\simeq} \ar[r]^-{\eta_X}
		& G_1(F_1(X)) \ar[d]^-{\simeq}
	\\
		G_1(G_2(F_2(F_1(X)))) \ar[r]^-{\cong}
		& G_1(G_2(F_2(F_1(G_1(F_1(X)))))).
	}\]
Then, $\eta_X$ is a weak equivalence.
\end{proof}
\begin{lem} \label{LemQuasiCounit}
Let $\epsilon$ be the counit of adjunction $F_1F_0 \dashv G_0G_1 : {\bf sSet} \rightarrow {\bf qTop}$.
For any quasitopological space $X$,  $G_1(\epsilon_X) : G_1(F_1(F_0(G_0(G_1(X))))) \rightarrow G_1(X)$ is a weak equivalence in {\bf pqTop}.
\end{lem}
\begin{proof}
We consider the following diagram:
	\[\xymatrix{
		F_0(G_0(G_1(X))) \ar[d]_-{\alpha} \ar[r]^-{\beta}
		& G_1(F_1(F_0(G_0(G_1(X))))) \ar[d]^-{G_1(\epsilon_X)}
	\\
		G_1(X) \ar@{=}[r]
		& G_1(X).
	}\]
$\alpha$ is a weak equivalence by Lemma \ref{LemPrequasiCounit}.
$\beta$ is a weak equivalence by Lemma \ref{LemSheafificationWeq}.
Then, $G_1(\epsilon_X)$ is a weak equivalence.
\end{proof}
Let $I_{\bf qTop}$, $J_{\bf qTop}$ and $W_{\bf qTop}$ be sets as the followings:
\[
\begin{array}{rcl}
I_{\bf qTop}
& := & F_1I_{\bf pqTop}, \\
J_{\bf qTop}
& := & F_1J_{\bf pqTop}, \\
W_{\bf qTop}
& := & G_1^{-1}W_{\bf pqTop},
\end{array}
\]
Now we are ready to prove that {\bf qTop} has a model structure.
\begin{thm}
{\bf qTop} has a cofibrantly generated model structure satisfying the followings:
\begin{itemize}
\item $I_{\bf qTop}$ generates cofibrations,
\item $J_{\bf qTop}$ generates acyclic cofibrations, and
\item $W_{\bf qTop}$ is the class of weak equivalences.
\end{itemize}
Furthermore, $F_1 \dashv G_1 : {\bf pqTop} \rightarrow {\bf qTop}$ is a Quillen equivalence.
\end{thm}
\begin{proof}
In order to apply Theorem \ref{ThmTransfer}, we show the followings:
\begin{description}
\item[(1)] $I_{\bf qTop}$ and $J_{\bf qTop}$ permit the small object argument.
\item[(2)]  We have the inclusion relation $\operatorname{cell}(J_{\bf qTop}) \subset W_{\bf qTop}$.
\end{description}\par
We will show the condition {\bf (1)}.
Any simplex $\Delta^n$ is small in {\bf pqTop} by Example \ref{ExSmallSheaf}.
$\partial \Delta^n$ and $\Lambda^n_i$ are finite polyhedra.
These are small in {\bf qTop} by Proposition \ref{PropSmallColimit}. \par
We will show the condition {\bf (2)}.
Any element $G_1(f) \in G_1(J_{\bf qTop})$ is isomorphic to $\Delta^{n-1} \overset{0}{\hookrightarrow} \Delta^{n-1} \times [0, 1]$ by Corollary \ref{CorFiniteEquivalence}.
This is isomorphic to a acyclic cofibration $F_0(\Delta^{n-1} \hookrightarrow \Delta^n)$ in {\bf pqTop}.
Then, we obtain the inclusion relation $\operatorname{cell}(G_1(J_{\bf qTop})) \subset W_{\bf pqTop}$.
We have the sequence of the inclusion relation $G_1(\operatorname{cell}(J_{\bf qTop})) \subset G_1(F_1(\operatorname{cell}(G_1(J_{\bf qTop})))) \subset W_{\bf pqTop}$ by Lemma \ref{LemSheafificationWeq}.
Then, we obtain the inclusion relation $\operatorname{cell}(J_{\bf qTop}) \subset W_{\bf qTop}$. \par
Therefpre, {\bf pqTop} has a cofibrantly generated model structure by Theorem \ref{ThmTransfer}. \par
We will show that the Quillen adjunction $F_1 \dashv G_1$ is a Quillen equivalence.
We consider the following diagram:
	\[\xymatrix{
		& {\bf qTop} \ar[dr]^-{F_2 \dashv G_2}
	\\
		{\bf sSet} \ar[ur]^{F_1F_0 \dashv G_0G_1} \ar[rr]_{|-| \dashv \operatorname{Sing}}^{\simeq}
		&& {\bf Top}.
	}\]
It is classically known that $|-| \dashv \operatorname{Sing}$ is a Quillen equivalence.
Then, the Quillen adjunction $F_1F_0 \dashv G_0G_1$ is a ``split monomorphism".
Strictly speaking, the unit of the derived adjunction $\mathbb{L}(F_1F_0) \dashv \mathbb{R}(G_0G_1)$ is a natural isomorphism.
The counit of the adjunction $F_1F_0 \dashv G_0G_1$ is a weak equivalence in {\bf pqTop} by Lemma \ref{LemQuasiCounit}.
This is a weak equivalence in {\bf qTop} by Definition of $W_{\bf qTop}$.
All objects in {\bf sSet} are cofibrant.
Then, the counit of the derived adjunction $\mathbb{L}(F_1F_0) \dashv \mathbb{R}(G_0G_1)$ is a natural isomorphism.
$\mathbb{L}(F_1F_0) \dashv \mathbb{R}(G_0G_1)$ is a categorical equivalence.
Therefore, $F_1F_0 \dashv G_0G_1$ is a Quillen equivalence.
$F_0 \dashv G_0$ is a Quillen equivalence by Theorem \ref{ThmModePqtop}.
Then, $F_1 \dashv G_1$ is a Quillen equivalence.
\end{proof}
	\subsubsection{Homotopy Theory in quasitopological spaces}
In this subsubsection, we will discuss Homotopy Theory in {\bf qTop}. \par
Any element $f \in J_{\bf qTop}$ is isomorphic to $\Delta^{n-1} \overset{0}{\hookrightarrow} \Delta^{n-1} \times [0, 1]$ by Corollary \ref{CorFiniteEquivalence}.
We call a fibration in {\bf qTop} a {\it Serre fibration}.
The model category {\bf qTop} has the following properties.
\begin{prop} \label{PropQTopModel}
\mbox{}
\begin{itemize}
\item Any quasitopological space $X$ is fibrant in {\bf qTop}.
\item {\bf qTop} is a right proper model category.
\item {\bf qTop} is a Cartesian closed model category.
\end{itemize}
\end{prop}
\begin{proof}
We will show that any quasitopological space $X$ is fibrant in {\bf qTop}.
Any element $f \in J_{\bf qTop}$ is isomorphic to $\Delta^{n-1} \overset{0}{\hookrightarrow} \Delta^{n-1} \times [0, 1]$ by Corollary \ref{CorFiniteEquivalence}.
This is a retract.
Then, the morphism $X \rightarrow {\bf 1}$ has the right lifting property with respect to $\Delta^{n-1} \overset{0}{\hookrightarrow} \Delta^{n-1} \times [0, 1]$.
This means that $X$ is fibrant in {\bf qTop}.
Furthermore, {\bf qTop} is a right proper model category by Proposition \ref{PropFibrantProper}. \par
We will show that {\bf qTop} is a Cartesian closed model category.
Any morphism $(\partial \Delta^n \times \Delta^m) \coprod_{(\partial \Delta^n \times \partial \Delta^m)} (\Delta^n \times \partial \Delta^m) \rightarrow \Delta^n \times \Delta^m$ is isomorphic to $\partial \Delta^{n+m} \rightarrow \Delta^{n+m}$ by Corollary \ref{CorFiniteEquivalence}.
This is a cofibration.
Any morphism $(\Lambda^n_i \times \Delta^m) \coprod_{(\Lambda^n_i \times \partial \Delta^m)} (\Delta^n \times \partial \Delta^m) \rightarrow \Delta^n \times \Delta^m$ is isomorphic to $\Lambda^{n+m}_i \rightarrow \Delta^{n+m}$ by Corollary \ref{CorFiniteEquivalence}.
This is an acyclic cofibration.
Therefore, {\bf qTop} is a Cartesian closed model category by Proposition \ref{PropCofibrantlyCartesian}.
\end{proof}
For any quasitopological space $X$, the simplicial set $G_0G_1(X)$ is a Kan complex because a right Quillen functor preserves fibrant objects.
We can define the {\it homotopy groups} of $X$ as the homotopy groups of the Kan complex $G_0G_1(X)$.
A morphism in {\bf qTop} is a weak equivalence if and only if all the maps induced between the homotopy groups are isomorphisms.
We call a weak equivalence in {\bf qTop} a {\it weak homotopy equivalence}.
\vspace{10pt}\\
Recall some other concepts and properties in Homotopy Theory in {\bf qTop}, as an analogue of that in {\bf Top}. \par
For any quasitopological space $X$, the exponential object $X^{\Delta^1}$ is a {\it path object} of $X$.
\begin{prop} \label{PropPath}
Let $X$ be a quasitopological space.
\begin{itemize}
\item The morphism $X^{\Delta^1} \rightarrow X^{\partial \Delta^1} \cong X \times X$ induced by $\partial \Delta^1 \rightarrow \Delta^1$ is a fibration.
\item The morphism $X^{\Delta^1} \overset{i}{\rightarrow} X^{\Delta^0} \cong X$($i = 0, 1$) induced by $\Delta^0 \overset{i}{\rightarrow} \Delta^1$ is an acyclic fibration.
\end{itemize}
\end{prop}
\begin{proof}
Let $f$ be the morphism $\partial \Delta^1 \rightarrow \Delta^1$, and let $g$ be the morphism $X \rightarrow {\bf 1}$.
$f$ is a cofibration.
$g$ is a fibration by Proposition \ref{PropQTopModel}.
Then the morphism $g^{\# f} : X^{\Delta^1} \rightarrow X^{\partial \Delta^1}$ is a fibration because {\bf qTop} is a Cartesian closed model category by Proposition \ref{PropQTopModel}
(cf. Definition \ref{DefModelCartesian}). \par
Let $f'$ be the morphism $\Delta^0 \overset{i}{\rightarrow} \Delta^1$.
$f'$ is an acyclic cofibration.
$g^{\# f'} : X^{\Delta^1} \overset{i}{\rightarrow} X^{\Delta^0}$ is an acyclic fibration because {\bf qTop} is a Cartesian closed model category by Proposition \ref{PropQTopModel}
(cf. Definition \ref{DefModelCartesian}).
\end{proof}
\begin{prop} \label{PropRetract}
Any acyclic cofibration in {\bf qTop} is a strong deformation retract.
\end{prop}
\begin{proof}
Take any acyclic cofibration $f : X \rightarrow Y$ in {\bf qTop}.
$f$ has a retraction $g : Y \rightarrow X$ because $X$ is fibrant
(cf. Proposition \ref{PropQTopModel}).
	\[\xymatrix{
		X \ar[d]_-{f} \ar[r]
		& X \ar[d]
	\\
		Y \ar[r] \ar@{.>}[ru]^-{g}
		& {\bf 1}.
}\]
There exists a homotopy $H$ from $f \circ g$ to $id_Y$ up to $X$ by Proposition \ref{PropPath}.
	\[\xymatrix{
		X \ar[d]_-{f} \ar[r]
		& Y^{\Delta^1} \ar[d]
	\\
		Y \ar[r]_-{id_Y \times (f \circ g)} \ar@{.>}[ru]^-{H}
		& Y \times Y.
}\]
This means that $f$ is a strong deformation retract.
\end{proof}
	\paragraph{The mapping track}
Any morphism $f$ in {\bf qTop} has a decomposition $f = f' \circ e_f$ such that the morphism $e$ is a weak equivalence and the morphism $p$ is a fibration.
The existence of such a decomposition is required by the axiom of a model category.
One of such decompositions, called the {\it mapping track}, can also be constructed in a specific way. \par
Let $f : X \rightarrow Y$ be a morphism in {\bf qTop}.
We define a quasitopological space $X_f$ as the following pull-back diagram:
	\[\xymatrix{
		X_f \ar[d] \ar[r]
		& Y^{[0, 1]} \ar[d]^-{0}
	\\
		X \ar[r]_-{f}
		& Y.
}\]
We call the quasitopological space $X_f$ the {\it mapping track} of $f$.
We define two morphisms $e_f : X \rightarrow X_f$ and $f' : X_f \rightarrow Y$ as the following diagram:
	\[\xymatrix{
		&& Y \ar[d]
	\\
		X \ar[r]^-{e_f} \ar@/^20pt/[rru]^-{f} \ar@{=}[rd]
		& X_f \ar[r] \ar[d]
		& Y^{[0, 1]} \ar[d]^-{0}
		& X_f \ar[r] \ar[rd]_-{f'}
		& Y^{[0, 1]} \ar[d]^-{1}
	\\
		& X \ar[r]_-{f}
		& Y
		&& Y.
	}\]
Then, we have the equality $f = f' \circ e_f$.
\begin{lem} \label{LemRetract}
Let $\Sigma := ([0, 1] \times \{ 1 \}) \cup (\{ 0 \} \times [0, 1])$.
The inclusion map $\Sigma \hookrightarrow [0, 1]^2$ has a strong deformation retraction $R_u : [0, 1]^2 \rightarrow [0, 1]^2$ satisfying the followings:
\begin{itemize}
\item $R_0 = id$ and $Im(R_1) \subset \Sigma$, 
\item $R_u |_\Sigma = id$, and
\item $R_u(s,0) \in [0, 1] \times \{ 0 \}$.
\end{itemize}
\end{lem}
\begin{proof}
We define two homotopies $R^1_u$ and $R^2_u$ as the following:
\[
\begin{array}{l}
R^1_u(s,t) := ((1-(1-t)u)s, t), \\
R^2_u(s,t) := (s, t) + \frac{(1-|s+t-1|)(1-|s-t|)}{2}u(-1, 1).
\end{array}
\]
These are the deformations illustrated in the following figure.
\vspace{5pt}
\[
\begin{array}{ccc}
\includegraphics[width=5cm]{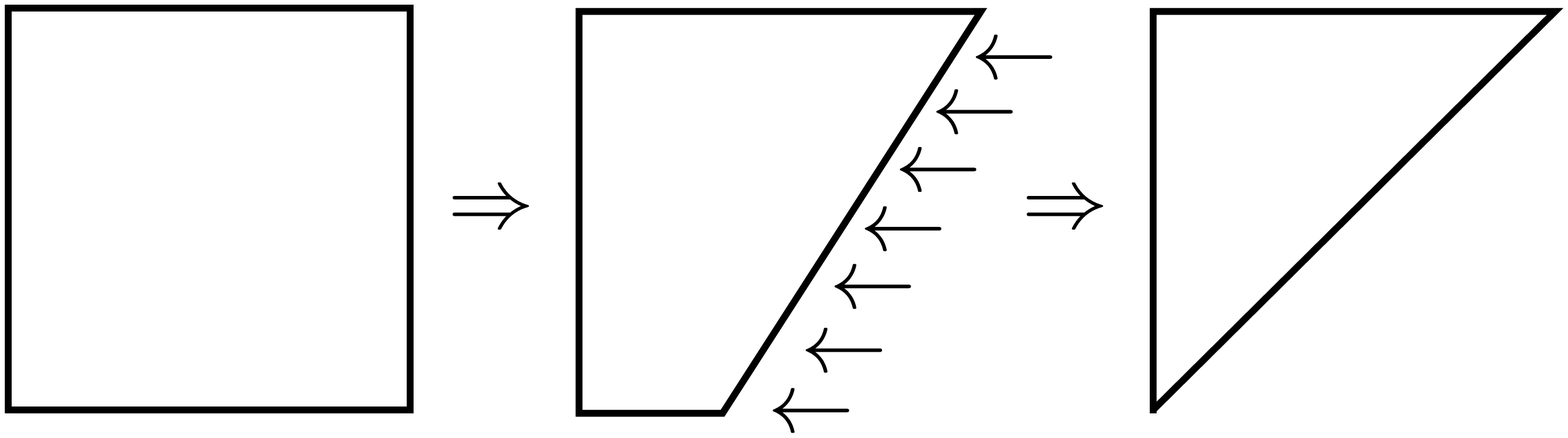}
&& \includegraphics[width=5cm]{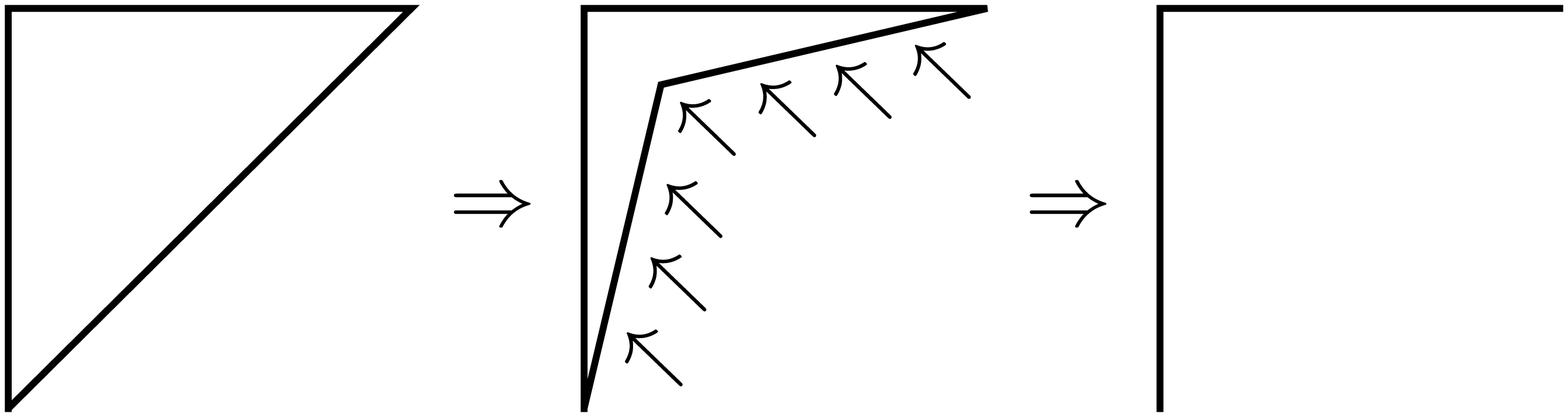} \\
R^1_u
&& R^2_u
\end{array}
\]
We define a homotopy $R_u$ as the following:
\[
R_u(s,t) := \left\{
\begin{array}{ll}
R^1_{2u}(s, t) & (\mbox{if } 0 \le u \le \frac{1}{2})\\
R^2_{2u-1}(R^1_1(s, t)) & (\mbox{if } \frac{1}{2} \le u \le 1).
\end{array}
\right.
\]
The above $R_u$ is the homotopy we want.
\end{proof}
\begin{prop}
\mbox{}
\begin{itemize}
\item The above morphism $e_f$ is a weak equivalence.
\item The above morphism $f'$ is a fibration.
\end{itemize}
\end{prop}
\begin{proof}
\setcounter{equation}{0}
$X_f \rightarrow X$ is an acyclic fibration because it is the pull-back of an acyclic fibration $Y^{[0, 1]} \overset{0}{\rightarrow} Y$
(cf. Proposition \ref{PropPath}).
In particular, $X_f \rightarrow X$ is a weak equivalence.
$e_f$ is a weak equivalence because the weak equivalence $X_f \rightarrow X$ is a left inverse of $e_f$. \par
In order to show that $f'$ is a fibration, we consider the following diagram:
\begin{equation}
\vcenter{\xymatrix{
		\Delta^{n-1} \ar[d] \ar[r]^-{\alpha}
		& X_f \ar[d]^-{f'}
	\\
		\Delta^{n-1} \times [0, 1] \ar[r]^-{\beta}
		& Y.
}}
\end{equation}
We define a morphism $\gamma_1$ as the following composition:
\[
\gamma_1 : \Delta^{n-1} \times [0, 1] \overset{pr_1}{\rightarrow} \Delta^{n-1} \overset{\alpha}{\rightarrow} X_f \rightarrow X
\]
Take the subset $\Sigma \hookrightarrow [0, 1]^2$ and the strong deformation retraction $R_u : [0, 1]^2 \rightarrow [0, 1]^2$ in Lemma \ref{LemRetract}.
Let $\alpha' : \Delta^{n-1} \times [0, 1] \rightarrow Y$ be the adjunct of the composite morphism $\Delta^{n-1} \overset{\alpha}{\rightarrow} X_f \rightarrow Y^{[0, 1]}$.
We define a morphism $\gamma'_2 : \Delta^{n-1} \times \Sigma \rightarrow Y$ as the following:
\[
\begin{array}{l}
\gamma'_2 |_{\Delta^{n-1} \times ([0, 1] \times \{ 1 \})} : \Delta^{n-1} \times ([0, 1] \times \{ 1 \}) \cong \Delta^{n-1} \times [0, 1] \overset{\beta}{\rightarrow} Y, \\
\gamma'_2 |_{\Delta^{n-1} \times (\{ 0 \} \times [0, 1])} : \Delta^{n-1} \times (\{ 0 \} \times [0, 1]) \cong \Delta^{n-1} \times [0, 1] \overset{\alpha'}{\rightarrow} Y.
\end{array}
\]
Let $\gamma_2 : \Delta^{n-1} \times [0, 1] \rightarrow Y^{[0, 1]}$ be the adjunct of the composite morphism $\gamma'_2 \circ R_1$.
We obtain a morphism $\gamma (= \gamma_1 \times \gamma_2) : \Delta^{n-1} \times [0, 1] \rightarrow X_f$.
The morphism $\gamma$ is a lift of the diagram (1).
Therefore, the morphism $f'$ has the right lifting property with respect to $J_{\bf qTop}$.
This means that $f'$ is a fibration.
\end{proof}
	\paragraph{HELP Lemma}
We will characterise a weak equivalence by HELP ({\it Homotopy Extension Lifting Property}).
\begin{prop}[HELP Lemma; cf. \cite{vogt2010help}] \label{PropQTopHelp}
For any morphism $f : X \rightarrow Y$ in {\bf qTop}, the following conditions are equivalent.
\begin{description}
\item[(1)] $f$ is a weak equivalence.
\item[(2)] $f$ has the {\it right Homotopy Extension Lifting Property} with respact to $I_{\bf qTop}$.
i.e. For a following commutative square, there exists a morphism $\gamma$ in the following diagram such that the upper triangle in the diagram is commutative and that the lower triangle in the diagram is commutative up to homotopy fixing $\partial \Delta^n$:
	\[\xymatrix{
		\partial \Delta^n \ar[r] \ar[d]
		& \cdot \ar[d]^-{f}
	\\
		\Delta^n \ar[r] \ar@{.>}[ur]^-{\gamma}
		& \cdot.
}\]
\end{description}
\end{prop}
\begin{proof}
\mbox{}\\
{\bf (1)} $\Rightarrow$ {\bf (2)}\\
Take a decomposition $f = p_f \circ i_f$ of $f$ such that $p_f$ is a fibration and that $i_f$ is an acyclic cofibration.
If $f$ is a weak equivalence, $p_f$ is acyclic.
Any morphism $\partial \Delta^n \hookrightarrow \Delta^n$ has the left lifting property with respect to $p_f$.
$i_f$ has a strong deformation retraction by Proposition \ref{PropRetract}.
Then, we obtain the condition {\bf (2)}. \\
{\bf (2)} $\Rightarrow$ {\bf (1)}\\
In order to show that $f$ is a weak equivalence, we will show that $f' : X_f \rightarrow Y$ is an acyclic fibration.
We consider the following diagram:
	\[\xymatrix{
		\partial \Delta^n \ar[d] \ar[r]^-{\alpha}
		& X_f \ar[rd]^-{f'} \ar[r]^-{pr_1}
		& X \ar[d]^-{f}
	\\
		\Delta^n \ar[rr]^-{\beta}
		&& Y
}\]
By the condition {\bf (2)}, there exist a morphism $\gamma' : \Delta^n \rightarrow X$ and a homotopy $H : \Delta^n \rightarrow Y^{[0, 1]}$ such that $H$ fixes $\partial \Delta^n$ and that the following diagram is commutative:
	\[\xymatrix{
		& \Delta^n \ar[ld]_-{f \circ \gamma'} \ar[rd]^-{\beta} \ar[d]^-{H}
	\\
		Y
		& Y^{[0, 1]} \ar[l]^-{0} \ar[r]_-{1}
		& Y
}\]
Let $\gamma := \gamma' \times H : \Delta^n \rightarrow X_f$.
$\gamma$ is a lift of the diagram $(\alpha, \beta)$.
Therefore, the morphism $f'$ has the right lifting property with respect to $I_{\bf qTop}$.
This means that $f'$ is an acyclic fibration.
Then, $f$ is a weak equivalence.
\end{proof}
	\section{Continuous sheaves}
In this section, we discuss continuous sheaves.
The sheaf of solutions of a partial differential relation is one of the most typical examples of a continuous sheaf.
A {\it flexible sheaf} is an important class of continuous sheaves.
Every continuous sheaf has a flexible sheaf callled the ``sheaf of formal sections''
(cf. Lemma \ref{LemGromov2}).
	\subsection{Continuous sheaves}
First, we define a continuous sheaf.
\begin{defi}
Let $B$ be a topological space. \par
A {\it continuous presheaf} on $B$ is a presheaf with value in {\bf qTop}, where {\bf qTop} is the category of quasitopological spaces
(cf. Definition \ref{DefQTop}).
${\bf PSh}(B; {\bf qTop})$ is the category of continuous presheaves on $B$. \par
A {\it continuous sheaf} on $B$ is a sheaf with value in {\bf qTop}.
${\bf Sh}(B; {\bf qTop})$ is the category of continuous sheaves on $B$.
\end{defi}
For any compact subset $K$ of $B$, we define a quasitopological space $\mathcal{F}(K)$ as 
\[
\mathcal{F}(K) := \displaystyle\lim_{\substack{\longrightarrow \\ U \supset K}} \mathcal{F}(U).
\]
\vspace{10pt}\\
Let $B_{top}$ be the site of open subsets of $B$.
Let $B_{top}^{tri}$ be the category, forgetting the Grothendieck topology of $B_{top}$.
$B_{top}^{tri}$ can be regarded as a site with the trivial topology.
Then, we obtain the following isomorphisms:
\[
\begin{array}{lclcl}
{\bf PSh}(B; {\bf qTop})
& \cong & {\bf Sh}(B_{top}^{tri} \times {\bf Sing\Delta}; {\bf Set})
& \cong & {\bf Sh}({\bf Sing\Delta}; {\bf PSh}(B; {\bf Set})), \\
{\bf Sh}(B; {\bf qTop})
& \cong & {\bf Sh}(B_{top} \times {\bf Sing\Delta}; {\bf Set})
& \cong & {\bf Sh}({\bf Sing\Delta}; {\bf Sh}(B; {\bf Set})).
\end{array}
\]
The category ${\bf Sh}(B_{top}^{tri} \times {\bf Sing\Delta}; {\bf Set})$ is embeded in the category ${\bf PSh}(B_{top} \times {\bf Sing\Delta}; {\bf Set})$ of presheaves.
Then, the embedding functor ${\bf Sh}(B_{top} \times {\bf Sing\Delta}; {\bf Set}) \rightarrow {\bf Sh}(B_{top}^{tri} \times {\bf Sing\Delta}; {\bf Set})$ has a left adjoint functor.
To repeat the same meaning, the forgetful functor ${\bf Sh}(B; {\bf qTop}) \rightarrow {\bf PSh}(B; {\bf qTop})$ has a left adjoint functor.
This left adjoint functor is called the {\it sheafification}. \par
Let $\mathcal{F}$ be a continuous (pre)sheaf.
For any open subset $U$ of $B$, $\mathcal{F}(U)$ is a quasitopological space.
For any closed subset $S$ of a simplex, $\mathcal{F}(U)[S]$ is a set.
We define a (pre)sheaf $\mathcal{F}[S]$ as $\mathcal{F}[S](U) := \mathcal{F}(U)[S]$.
\begin{prop} \label{PropStalk}
Let $\mathcal{F}$ be a continuous (pre)sheaf.
For any point $x \in B$ and any closed subset $S$ of a simplex, we have an isomorphism $\mathcal{F}_x[S] \cong \mathcal{F}[S]_x$ of stalks.
\end{prop}
\begin{proof}
It is obvious by Proposition \ref{PropFilter}.
\end{proof}
\begin{prop} \label{PropShefification}
Let $\mathcal{F}$ be a continuous presheaf, and let $\widehat{\mathcal{F}}$ be the sheafification of $\mathcal{F}$.
For any closed subset $S$ of a simplex, the sheaf $\widehat{\mathcal{F}}[S]$ is the sheafification of the presheaf $\mathcal{F}[S]$.
\end{prop}
\begin{proof}
For any point $x \in B$, we have an isomorphism $\mathcal{F}_x \cong \widehat{\mathcal{F}}_x$ of stalks.
This means that, for any closed subset $S$ of a simplex, we have an isomorphism $\mathcal{F}[S]_x \cong \widehat{\mathcal{F}}[S]_x$ of stalks, by Proposition \ref{PropStalk}.
Then, the sheaf $\widehat{\mathcal{F}}[S]$ is the sheafification of the presheaf $\mathcal{F}[S]$.
\end{proof}
\begin{prop} \label{PropSectionwiseStalkwise}
Let $f : \mathcal{F} \rightarrow \mathcal{G}$ be a morphism between continuous presheaves.
If $f$ is a sectionwise weak equivalence (i.e. $f_U : \mathcal{F}(U) \rightarrow \mathcal{G}(U)$ is a weak equivalence for any open set $U$), then $f$ is a stalkwise weak equivalence (i.e. $f_x : \mathcal{F}_x \rightarrow \mathcal{G}_x$ is a weak equivalence for any point $x$),
\end{prop}
\begin{proof}
Suppose that $f$ is a sectionwise weak equivalence.
We will show that $f_x$ has the {\it right Homotopy Extension Lifting Property} with respect to $I_{\bf qTop}$
(cf.  Proposition \ref{PropQTopHelp}).
i.e. We will show that, for a following commutative square, there exists a morphism $\gamma$ in the following diagram such that the upper triangle in the diagram is commutative and that the lower triangle in the diagram is commutative up to homotopy fixing $\partial \Delta^n$:
	\[\xymatrix{
		\partial \Delta^n \ar[r] \ar[d]
		& \mathcal{F}_x \ar[d]^-{f_x}
	\\
		\Delta^n \ar[r] \ar@{.>}[ur]^-{\gamma}
		& \mathcal{G}_x.
}\]
By the discussion in Section 1.1.4, there exists a small enough open neighbourhood $U$ of $x$ such that the following diagram is commutative:
	\[\xymatrix{
		\partial \Delta^n \ar[r] \ar[d]
		& \mathcal{F}(U) \ar[d]^-{f_U}
	\\
		\Delta^n \ar[r]
		& \mathcal{G}(U).
}\]
$f_U$ has the {\it right Homotopy Extension Lifting Property} with respect to $I_{\bf qTop}$ by Proposition \ref{PropQTopHelp}.
Then, $f_x$ also has the {\it right Homotopy Extension Lifting Property} with respect to $I_{\bf qTop}$.
$f_x$ is a weak equivalence by Proposition \ref{PropQTopHelp}.
Therefore, $f$ is a stalkwise weak equivalence.
\end{proof}
	\subsection{Flexible sheaves}
Before defining a flexible sheaf, we prepare a notation.
\begin{defi} \label{DefCompactPairMorphism}
Let $B$ be a topological space.
Let $f : \mathcal{F} \rightarrow \mathcal{G}$ be a morphism between continuous sheaves on $B$.
For compact subsets $K$ and $L$ of $B$ with $K \subset L$, denote the following morphism $\mathcal{F}(L) \rightarrow \mathcal{F}(K) \times_{\mathcal{G}(K)} \mathcal{G}(L)$ by $f_{K \subset L}$:
	\[\xymatrix{
		\mathcal{F}(L) \ar@/_15pt/[ddr] \ar@/^15pt/[drr]^-{f_L} \ar[dr]^{f_{K \subset L}}
	\\
		& \mathcal{F}(K) \times_{\mathcal{G}(K)} \mathcal{G}(L) \ar[d] \ar[r]
		& \mathcal{G}(L) \ar[d]
	\\
		& \mathcal{F}(K) \ar[r]^-{f_K}
		& \mathcal{G}(K).
	}\]
\end{defi}
\begin{defi} \label{DefFlexible}
Let $B$ be a topological space.
A morphism $f : \mathcal{F} \rightarrow \mathcal{G}$ between continuous sheaves on $B$ is a {\it flexible extension} if, for any compact subsets $K$ and $L$ of $B$ with $K \subset L$, the induced morphism $f_{K \subset L}$ is a fibration.
A continuous sheaf $\mathcal{F}$ is {\it flexible} if the unique morphism $\mathcal{F} \rightarrow {\bf 1}$ is a flexible extension, where {\bf 1} is the terminal object.
\end{defi}
A continuous sheaf $\mathcal{F}$ is a flexible sheaf if and only if, for any compact subsets $K$ and $L$ of $B$ with $K \subset L$, the restriction morphism $\mathcal{F}(L) \rightarrow \mathcal{F}(K)$ is a fibration.
	\subsection{The sheaf of formal sections}
Recall that a {\it partial differential relation} is a subbundle of a jet bundle and that a {\it formal solution} of a partial differential relation $R$ is a continuous section of the fibre bundle $R$
(cf. \cite{eliashberg2002introduction, gromov2013partial}).
In this subsection, we define the {\it sheaf of formal sections} of a continuous sheaf.
\begin{defi}
For a continuous sheaf $\mathcal{F}$ on a topological space$B$, we define a continuous presheaf $\mathcal{F}^\square$ on $B$ as the following:
\[
\mathcal{F}^\square(U) := \mathcal{F}(U)^{\operatorname{Sing^q}(U)}.
\]
The {\it sheaf of formal sections} of $\mathcal{F}$ is the sheafification $\mathcal{F}^\ast$ of $\mathcal{F}^\square$. \par
The morphisms $\mathcal{F}(U) \rightarrow \mathcal{F}(U)^{\operatorname{Sing^q}(U)}$ induced the canonnical projections $\mathcal{F}(U) \times \operatorname{Sing^q}(U) \rightarrow \mathcal{F}(U)$ define a morphism $\mathcal{F} \rightarrow \mathcal{F}^\square$.
The {\it diagonal morphism} is the induced morphism $\Delta : \mathcal{F} \rightarrow \mathcal{F}^\ast$ by the above morphism $\mathcal{F} \rightarrow \mathcal{F}^\square$.
\end{defi}
\begin{rem}
Compared with a formal solution of a partial differential relation, the above definition may seem strange.
In fact, it is more appropriate to call the above $\mathcal{F}^\ast$ a {\it sheaf of fibrewise holonomic sections} (cf. \cite[page 24]{eliashberg2002introduction}).
\end{rem}
We can prove the claim {\bf (1)} of Lemma \ref{Gromov2}.
First, we define the properties that $B$ must satisfy.
\begin{defi}
A topological space $B$ is {\it strongly locally contractible} if any point $x \in B$ has an open neighbourhood basis $\mathcal{N}$ such that, for any open neighbourhood $U \in \mathcal{N}$, the inclusion map $\{ x \} \hookrightarrow U$ is a strong deformation retract.
\end{defi}
\begin{lem}[Gromov \cite{gromov2013partial}] \label{LemGromov1}
Let $B$ be a strongly locally contractible space.
For any continuous sheaf $\mathcal{F}$ on $B$, the diagonal morphism $\Delta : \mathcal{F} \rightarrow \mathcal{F}^\ast$ is a stakwise weak equivalence.
\end{lem}
\begin{proof}
We will show that $\Delta_x$ has the {\it right Homotopy Extension Lifting Property} with respect to $I_{\bf qTop}$
(cf.  Proposition \ref{PropQTopHelp}).
i.e. We will show that, for a following commutative square, there exists a morphism $\gamma$ in the following diagram such that the upper triangle in the diagram is commutative and that the lower triangle in the diagram is commutative up to homotopy fixing $\partial \Delta^n$:
	\[\xymatrix{
		\partial \Delta^n \ar[r] \ar[d]
		& \mathcal{F}_x \ar[d]^-{\Delta_x}
	\\
		\Delta^n \ar[r] \ar@{.>}[ur]^-{\gamma}
		& \mathcal{F}^\ast_x \cong \mathcal{F}^\square_x.
}\]
By the discussion in Section 1.1.4, there exists a small enough open neighbourhood $U$ of $x$ such that the following diagram is commutative:
	\[\xymatrix{
		\partial \Delta^n \ar[r]^-{\alpha} \ar[d]
		& \mathcal{F}(U) \ar[d]
	\\
		\Delta^n \ar[r]_-{\beta}
		& \mathcal{F}(U)^U.
}\]
We can assume that the inclusion map $\{ x \} \hookrightarrow U$ has a strong deformation retraction $p_u : U \rightarrow U$ because $B$ is strongly locally contractible.
$p_u$ induces a strong deformation retraction $H_u : \mathcal{F}(U)^U \rightarrow \mathcal{F}(U)^U$ of the morphism $\mathcal{F}(U) \rightarrow \mathcal{F}(U)^U$ in the above diagram.
Let $\gamma := H_1 \circ \beta : \Delta^n \rightarrow \mathcal{F}(U)$.
Then, the upper triangle is commutative, and the lower triangle is commutative up to homotopy $H_u \circ \beta$.
The homotopy $H_u \circ \beta$ fixes $\partial \Delta^n$.
Therefore, $\Delta_x$ has the {\it right Homotopy Extension Lifting Property} with respect to $I_{\bf qTop}$.
$\Delta_x$ is a weak equivalence by Proposition \ref{PropQTopHelp}.
\end{proof}
Let $B$ be a topological space, and let $\mathcal{F}$ be a continuous sheaf on $B$.
Let $\phi : S \times B \rightarrow T$ be a continuous map, where $S$ and $T$ are closed subsets of simplexes.
We will define a morphism $\phi^\ast : \mathcal{F}^\ast[T] \rightarrow \mathcal{F}^\ast[S]$ between sheaves of sets.
For any open subset $U$ of $B$, we define a map $\phi^\dag_U : \mathcal{F}^\square[T](U) \rightarrow \mathcal{F}^\square[S](U)$ as the following diagram:
	\[\xymatrix{
		\mathcal{F}^\square[T](U) \ar[r]^-{\phi^\dag_U}
		& \mathcal{F}^\square[S](U)
	\\
		{\bf qTop}(\operatorname{Sing^q}(T), \mathcal{F}(U)^{\operatorname{Sing^q}(U)}) \ar[u]_-{\cong}^-{\mbox{Yoneda iso.}}
		& {\bf qTop}(\operatorname{Sing^q}(S), \mathcal{F}(U)^{\operatorname{Sing^q}(U)}), \ar[u]^-{\cong}_-{\mbox{Yoneda iso.}}
	\\
		{\bf qTop}(\operatorname{Sing^q}(T \times U), \mathcal{F}(U)) \ar[u]_-{\cong}^-{\mbox{adjunction}} \ar[r]^-{\operatorname{Sing^q}(\hat{\phi})^\ast}
		& {\bf qTop}(\operatorname{Sing^q}(S \times U), \mathcal{F}(U)), \ar[u]^-{\cong}_-{\mbox{adjunction}}
	}\]
where $\hat{\phi} : S \times U \rightarrow T \times U$ is a continuous map defined as $\hat{\phi}(p, x) := (\phi(p, x), x)$.
These maps $\phi^\dag_U$ define a morphism $\phi^\dag : \mathcal{F}^\square[T] \rightarrow \mathcal{F}^\square[S]$ between presheaves of sets.
This morphism $\phi^\dag$ induces a morphism $\phi^\ast : \mathcal{F}^\ast[T] \rightarrow \mathcal{F}^\ast[S]$ between sheaves of sets
(cf. Proposition \ref{PropShefification}). \par
For any open subset $U$ of $B$, the map $\phi^\ast_U : \mathcal{F}^\ast[T](U) \rightarrow \mathcal{F}^\ast[S](U)$ induces the map ${\bf qTop}(\operatorname{Sing^q}(T), \mathcal{F}^\ast(U)) \rightarrow {\bf qTop}(\operatorname{Sing^q}(S), \mathcal{F}^\ast(U))$ as the following diagram:
	\[\xymatrix{
		\mathcal{F}^\ast[T](U) \ar[r]^-{\phi^\ast_U}
		& \mathcal{F}^\ast[S](U)
	\\
		{\bf qTop}(\operatorname{Sing^q}(T), \mathcal{F}^\ast(U)) \ar[u]_-{\cong}^-{\mbox{Yoneda iso.}} \ar[r]
		& {\bf qTop}(\operatorname{Sing^q}(S), \mathcal{F}^\ast(U)). \ar[u]^-{\cong}_-{\mbox{Yoneda iso.}}
	}\]
We denote this map by the same symbol, $\phi^\ast_U$.
	\section{Result 1 : Embedding to a model category}
The goal of this section is to prove the following theorem.
\begin{thm} \label{Thm1}
There exists a right proper model category ${\bf PSh}_\ast(\tilde{\mathscr{U}}_B; {\bf qTop})$ and a fully faithful and left exact embedding $\iota : {\bf Sh}(B; {\bf qTop}) \rightarrow {\bf PSh}_\ast(\tilde{\mathscr{U}}_B; {\bf qTop})$ such that the followings hold. \par
Let $f : \mathcal{F} \rightarrow \mathcal{G}$ be a morphism in ${\bf Sh}(B; {\bf qTop})$.
\begin{itemize}
\item $\iota(f)$ is a weak equivalence \\
$\Leftrightarrow$ $f$ is a stalkwise weak equivalence
(i.e. $f_x : \mathcal{F}_x \rightarrow \mathcal{G}_x$ is a weak equivalence for any $x \in B$).
\item $\iota(f)$ is a fibration \\
$\Leftrightarrow$ $f$ is a flexible extension
(cf. Definition \ref{DefFlexible}).
\end{itemize}
In particular, for a continuous sheaf $\mathcal{F}$ on $B$, 
\begin{itemize}
\item $\iota(\mathcal{F})$ is fibrant \\
$\Leftrightarrow$ $\mathcal{F}$ is flexible
(cf. Definition \ref{DefFlexible}).
\end{itemize}
\end{thm}
\begin{proof}
This follows immediately from Proposition \ref{PropEmbedding} and Theorem \ref{ThmPointwiseFlexible}.
\end{proof}
As a first step, we want to characterise flexible extensions by the right lifting property.
First, suppose that, for any compact subset $K$ of $B$, the inverse functor $i^\ast : {\bf Sh}(K; {\bf qTop}) \rightarrow {\bf Sh}(B; {\bf qTop})$ had a left adjoint functor $i_! : {\bf Sh}(B; {\bf qTop}) \rightarrow {\bf Sh}(K; {\bf qTop})$.
The unique map $\mu : K \rightarrow {\bf 1}$ induces an adjunction $\mu^\ast \dashv \mu_\ast : {\bf qTop} \rightarrow {\bf Sh}(K; {\bf qTop})$.
We obtain an adjunction $i_!\mu^\ast \dashv \mu_\ast i^\ast : {\bf qTop} \rightarrow {\bf Sh}(B; {\bf qTop})$.
For a quasitopological space $X$, we define a continuous sheaf $X_{!K}$ as $i_!\mu^\ast(X)$.
For a continuous sheaf $\mathcal{F}$, we have the equality $\mu_\ast i^\ast(\mathcal{F}) = \mathcal{F}(K)$.
Then, for any morphism $f : \mathcal{F} \rightarrow \mathcal{G}$ between continuous sheaves and any compact subsets $K$ and $L$ of $B$ with $K \subset L$, the following conditions are equivalent.
\begin{itemize}
\item The following square has a lift:
	\[\xymatrix{
		\Lambda^n_i \ar[r] \ar[d]
		& \mathcal{F}(L) \ar[d]^-{f_{K \subset L}}
	\\
		\Delta^n \ar[r]
		& \mathcal{F}(K) \times_{\mathcal{G}(K)} \mathcal{G}(L).
	}\]
\item The following square has a lift:
	\[\xymatrix{
		(\Lambda^n_i)_{!L} \cup_{(\Lambda^n_i)_{!K}} \Delta^n_{!K} \ar[r] \ar[d]
		& \mathcal{F} \ar[d]^-{f}
	\\
		\Delta^n_{!L} \ar[r]
		& \mathcal{G}.
	}\]
\end{itemize}
Then, a morphism $f$ between continuous sheaves is a flexible extension if and only if it has the reight lifting property with respect to $(\Lambda^n_i)_{!L} \cup_{(\Lambda^n_i)_{!K}} \Delta^n_{!K} \rightarrow \Delta^n_{!L}$.
We have characterized flexible extensions by the right lifting property in this situation. \par
However, the supposition of the existence of a left adjoint functor $i_!$ is too strong.
To avoid this problem, we define a category $\tilde{\mathscr{U}}_B$ (cf. Section 3.1) by extending the site associated with the topological space $B$, and consider the category of presheaves on $\tilde{\mathscr{U}}_B$.
Furthermore, for a technical reason (cf. Proposition \ref{PropAugmented}), we require the condition $\mathcal{F}(\emptyset) = {\bf 1}$ on a presheaf.
We call such a presheaf an {\it augmented presheaf}
(cf. Definition \ref{DefAugmented}).
${\bf PSh}_\ast(\tilde{\mathscr{U}}_B; {\bf qTop})$ is the category of augmented presheaves on $\tilde{\mathscr{U}}_B$.
\vspace{10pt}\\
The next step is to define a model structure $\mathcal{M}$ on ${\bf PSh}_\ast(\tilde{\mathscr{U}}_B; {\bf qTop})$, satisfying the following two conditions.
\begin{description}
\item[(W)] $f$ is a weak equivalence \\
$\Leftrightarrow$ $f$ is a stalkwise weak equivalence.
\item[(F)] $f$ is a fibration \\
$\Leftrightarrow$ $f$ is a flexible extension.
\end{description}
To construct $\mathcal{M}$, we define two model structures, $\mathcal{M}_1$ and $\mathcal{M}_2$.
$\mathcal{M}_1$ satisfies the above condition {\bf (F)}.
$\mathcal{M}_2$ satisfies the above condition {\bf (W)}.
These are mixed to make $\mathcal{M}$.
We call the model structure $\mathcal{M}_1$ the {\it compact flexible model structure}
(cf. Section 3.2).
We call the model structure $\mathcal{M}_2$ the {\it pointwise projective model structure}
(cf. Section 3.3).
We call the model structure $\mathcal{M}$ the {\it pointwise flexible model structure}
(cf. Section 3.4).
	\subsection{The base category}
We define a set $\tilde{\mathscr{U}}_B$ as the following:
\[
\tilde{\mathscr{U}}_B := \{ U \cap K \, | \, \mbox{$U (\subset B)$ is open, and $K (\subset B)$ is compact.} \}.
\]
$\tilde{\mathscr{U}}_B$ is an ordered set.
Then, $\tilde{\mathscr{U}}_B$ is a small category. \par
We define an augmented presheaf.
\begin{defi} \label{DefAugmented}
A presheaf $\mathcal{F}$ is {\it augmented} if $\mathcal{F}(\emptyset) = {\bf 1}$, where $\emptyset$ is the initial object and {\bf 1} is the terminal object.
\end{defi}
Grothendieck's original definition of presheaf is a contravariant functor from the set of {\it non-empty} open sets to a category
(cf. \cite{grothendieck1957quelques}).
A presheaf $\mathcal{F}$ in Grothendieck's sense is regarded as an augmented presheaf by defining $\mathcal{F}(\emptyset) := {\bf 1}$.
(Compare the definition of the {\it augmented simplicial set}.) \par
The advantage of a presheaf being augmented is the following property.
\begin{prop} \label{PropAugmented}
Let $f : \mathcal{F} \rightarrow \mathcal{G}$ be a morphism between continuous sheaves on $B$.
If $\mathcal{F}$ and $\mathcal{G}$ are augmented, then $f_{\emptyset \subset K}$ is isomorphic to $f_K : \mathcal{F}(K) \rightarrow \mathcal{G}(K)$ for any compact subset $K$ of $B$.
\end{prop}
\begin{proof}
It is obvious by definition of $f_{\emptyset \subset K}$
(cf. Definition \ref{DefCompactPairMorphism}).
\end{proof}
${\bf PSh}_\ast(\tilde{\mathscr{U}}_B; {\bf qTop})$ is the category of augmented presheaves with value in {\bf qTop} on $\tilde{\mathscr{U}}_B$.
We define a functor $\iota : {\bf Sh}(B; {\bf qTop}) \rightarrow {\bf PSh}_\ast(\tilde{\mathscr{U}}_B; {\bf qTop})$ as the following:
\[
\iota(\mathcal{F})(A) := \displaystyle\lim_{\substack{\longrightarrow \\ U \supset A}} \mathcal{F}(U).
\]
\begin{prop} \label{PropEmbedding}
The above functor $\iota$ is fully faithful and left exact.
\end{prop}
\begin{proof}
Let $\iota_1 : {\bf Sh}(B; {\bf qTop}) \rightarrow {\bf PSh}_\ast(B; {\bf qTop})$ be the forgetful functor, and let $\iota_2 : {\bf PSh}_\ast(B; {\bf qTop}) \rightarrow {\bf PSh}_\ast(B_{cpt}; {\bf qTop})$ be a functor defined as the following:
\[
\iota_2(\mathcal{F})(A) := \displaystyle\lim_{\substack{\longrightarrow \\ U \supset A}} \mathcal{F}(U).
\]
We have the equality $\iota = \iota_2 \circ \iota_1$.
$\iota_1$ is fully faithful and left exact.
We will show that $\iota_2$ is fully faithful and left exact.
We define a functor $\nu : {\bf PSh}_\ast(B_{cpt}; {\bf qTop}) \rightarrow {\bf PSh}_\ast(B; {\bf qTop})$ as the following:
\[
\nu(\mathcal{F})(U) := \mathcal{F}(U).
\]
Then, there exists an adjunction $\iota_2 \dashv \nu$.
The unit of $\iota_2 \dashv \nu$ is a natural isomorphism.
Then, $\iota_2$ is fully faithful.
Any finite limit and any filtered colimit are commutative in {\bf qTop} by the objectwise calculation
(cf. Proposition \ref{PropFilter}).
Then, $\iota_2$ is left exact.
Therefore, $\iota (= \iota_2 \circ \iota_1)$ is fully faithful and left exact.
\end{proof}
	\subsection{The compact flexible model structure}
In this subsection, we will prove the following theorem.
\begin{thm}[the compact flexible model structure] \label{ThmCompactFlexible}
The category ${\bf PSh}_\ast(\tilde{\mathscr{U}}_B; {\bf qTop})$ has a model structure satisfying the followings:
\begin{itemize}
\item $f$ is a weak equivalence \\
$\Leftrightarrow$ $f_K : \mathcal{F}(K) \rightarrow \mathcal{G}(K)$ is a weak equivalence for any compact subset $K$ of $B$.
\item $f$ is a fibration \\
$\Leftrightarrow$ $f$ is a flexible extension
(cf. Definition \ref{DefFlexible}).
\end{itemize}
\end{thm}
To prove the above theorem, we apply {\it Recognition Theorem} (cf. Theorem \ref{ThmRecognition}) of a cofibrantly generated model structure. \par
First, we prepare a notation.
For any quasitopological space $X$ and any compact subset $K$ of $B$, we define an object $X_{!K}$ in ${\bf PSh}_\ast(\tilde{\mathscr{U}}_B; {\bf qTop})$ as the following:
\[
X_{!K}(A) := \left\{
\begin{array}{ll}
X & (\mbox{if }\emptyset \neq A \subset K) \\
\emptyset & (\mbox{if }A \not\subset K) \\
{\bf 1} & (\mbox{if }A = \emptyset).
\end{array}
\right.
\]
For any morphism $f : X \rightarrow Y$ in {\bf qTop} and any compact subsets $K$ and $L$ of $B$ with $K \subset L$, denote the following morphism $X_{!L} \cup_{X_{!K}} Y_{!K} \rightarrow Y_{!L}$ by $f_{!(K \subset L)}$:
	\[\xymatrix{
		&& Y_{!L}
	\\
		X_{!L} \ar[r] \ar@/^15pt/[urr]^-{f_{!L}}
		& X_{!L} \cup_{X_{!K}} Y_{!K} \ar[ur]^-{f_{!(K \subset L)}}
	\\
		X_{!K} \ar[r]^-{f_{!K}'} \ar[u]
		& Y_{!K}. \ar[u] \ar@/_15pt/[uur]
	}\]
We define two index sets $\mathfrak{K}$ and $\mathfrak{K}^{(2)}$ as the followings:
\[
\begin{array}{lcl}
\mathfrak{K} & := & \{ K \subset B : \mbox{compact} \}, \\
\mathfrak{K}^{(2)} & := & \{ (K, L) \, | \, K, L \in \mathfrak{K}, K \subset L \}.
\end{array}
\]
We define two sets $I_{cpt}$ and $J_{cpt}$ and a class $W_{cpt}$ of some morphisms in ${\bf PSh}_\ast(\tilde{\mathscr{U}}_B; {\bf qTop})$ as the followings:
\[
\begin{array}{lcl}
I_{cpt} & := & \{ f_{!(K \subset L)} \, | \, f \in I_{\bf qTop}, (K, L) \in \mathfrak{K}^{(2)} \}, \\
J_{cpt} & := & \{ f_{!(K \subset L)} \, | \, f \in J_{\bf qTop}, (K, L) \in \mathfrak{K}^{(2)} \}, \\
W_{cpt} & := & \{ f : \mathcal{F} \rightarrow \mathcal{G} \, | \, f_K : \mathcal{F}(K) \rightarrow \mathcal{G}(K) \mbox{ is a weak equivalence for }^\forall K \in \mathfrak{K} \}.
\end{array}
\]
Then, a morphism $f$ in ${\bf PSh}_\ast(\tilde{\mathscr{U}}_B; {\bf qTop})$ is a flexible extension if and only if $f$ has the right lifting property with respect to $J_{cpt}$
(cf. the preamble to Section 3).
\begin{proof}[Proof of Theorem \ref{ThmCompactFlexible}]
In order to apply Theorem \ref{ThmRecognition}, we show the followings:
\begin{description}
\item[(1)] We have the inclusion relation $\operatorname{Cof}(J_{cpt}) \subset W_{cpt} \cap \operatorname{Cof}(I_{cpt})$.
i.e. We have the followings:
	\begin{description}
	\item[(1.1)] $\operatorname{Cof}(J_{cpt}) \subset \operatorname{Cof}(I_{cpt})$, and
	\item[(1.2)] $\operatorname{Cof}(J_{cpt}) \subset W_{cpt}$.
	\end{description}
\item[(2)] We have the equality $\operatorname{rlp}(I_{cpt}) = \operatorname{rlp}(J_{cpt}) \cap W_{cpt}$.
i.e. We have the followings:
	\begin{description}
	\item[(2.1)] $\operatorname{rlp}(I_{cpt}) \subset \operatorname{rlp}(J_{cpt})$,
	\item[(2.2)] $\operatorname{rlp}(I_{cpt}) \subset W_{cpt}$, and
	\item[(2.3)] $\operatorname{rlp}(I_{cpt}) \supset \operatorname{rlp}(J_{cpt}) \cap W_{cpt}$.
	\end{description}
\item[(3)] If two of the three morphisms $f$, $g$ and $g \circ f$ are contained in $W_{cpt}$, then the third is also contained in $W_{cpt}$.
\item[(4)] $I_{cpt}$ and $J_{cpt}$ permit the small object argument.
\end{description}
We will prove these in order.
\vspace{2pt}\\
\underline{Proof of the condition {\bf (1.1)}}
\vspace{2pt}\\
We have the inclusion relation $J_{cpt} \subset \operatorname{cell}(I_{cpt})$ by the decomposition of $f_{!(K\subset L)} \in J_{cpt}$represented by the following diagram:
	\[\xymatrix{
		(\Lambda^n_i)_{!L} \cup_{(\Lambda^n_i)_{!K}} \Delta^n_{!K} \ar[r] \ar@/^15pt/[rrr]^-{f_{!(K\subset L)}}
		& \partial \Delta^n_{!L} \cup_{\partial \Delta^n_{!K}} \Delta^n_{!K} \ar[rr]_-{h_{!(K\subset L)}}
		&& \Delta^n_{!L}
	\\
		\partial \Delta^{n-1}_{!L} \cup_{\partial \Delta^{n-1}_{!K}} \Delta^{n-1}_{!K} \ar[u] \ar[r]^-{g_{!(K\subset L)}}
		& \Delta^{n-1}_{!L}, \ar[u]
	}\]
where $g_{!(K \subset L)}$ and $h_{!(K\subset L)}$ belong to $I_{cpt}$ and the left square is a push-out.
Then, we obtain the inclusion relation $\operatorname{cell}(J_{cpt}) \subset \operatorname{cell}(I_{cpt})$.
Then, we obtain the inclusion relation $\operatorname{Cof}(J_{cpt}) \subset \operatorname{Cof}(I_{cpt})$.
\vspace{2pt}\\
\underline{Proof of the condition {\bf (1.2)}}
\vspace{2pt}\\
For any quasitopological space $X$ and any compact subset $K$ of $B$, we have the equality $X_{!K} = X_{!B} \times {\bf 1}_{!K}$.
For any morphism $X \rightarrow Y$ belonging to $J_{\bf qTop}$, $X_{!B} \rightarrow Y_{!B}$ is a sectionwise acyclic cofibration.
(i.e. For any object $A$ in $\tilde{\mathscr{U}}_B$, the morphism $X_{!B}(A) \rightarrow Y_{!B}(A)$ in {\bf qTop} is an acyclic cofibration.)
For any compact subsets $K$ and $L$ of $B$ with $K \subset L$, the morphism ${\bf 1}_{!K} \rightarrow {\bf 1}_{!L}$ is a sectionwise cofibration.
Then, the morphism $X_{!L} \coprod_{X_{!K}} Y_{!K} \rightarrow Y_{!L}$ is a sectionwise acyclic cofibration because {\bf qTop} is a Cartesian closed model category.
This means that any morphism $f_{!(K \subset L)}$ belonging to $J_{cpt}$ is a sectionwise acyclic cofibration.
Any push-out, any transfinite composition and any retract preserve a sectionwise acyclic cofibration by the sectionwise calculation.
Then, any morphism $f$ belonging to $\operatorname{Cof}(J_{cpt})$ is a sectionwise acyclic cofibration.
In particular, the morphism $f_K$ is a weak equivalence for any $K \in \mathfrak{K}$.
Therefore, we have the inclusion relation $\operatorname{Cof}(J_{cpt}) \subset W_{cpt}$.
\vspace{2pt}\\
\underline{Proof of the condition {\bf (2.1)}}
\vspace{2pt}\\
We have the followings:
\[
\begin{array}{rl}
& f \in \operatorname{rlp}(I_{cpt}) \\
\Leftrightarrow & \mbox{$f_{K \subset L}$ is an acyclic fibration for $^\forall (K, L) \in \mathfrak{K}^{(2)}$} \\
\Rightarrow & \mbox{$f_{K \subset L}$ is a fibration for $^\forall (K, L) \in \mathfrak{K}^{(2)}$} \\
\Leftrightarrow & f \in \operatorname{rlp}(J_{cpt}).
\end{array}
\]
We obtain the inclusion relation $\operatorname{rlp}(I_{cpt}) \subset \operatorname{rlp}(J_{cpt})$.
\vspace{2pt}\\
\underline{Proof of the condition {\bf (2.2)}}
\vspace{2pt}\\
Take any element $f \in \operatorname{rlp}(I_{cpt})$.
The induced morphism $f_{\emptyset \subset K}$ is an acyclic fibration for any compact subset $K$ of $B$.
In particular, this is a weak equivalence.
Then, $f_K$ is a weak equivalence by Proposition \ref{PropAugmented}.
Therefore, we obtain the inclusion relation $\operatorname{rlp}(I_{cpt}) \subset W_{cpt}$.
\vspace{2pt}\\
\underline{Proof of the condition {\bf (2.3)}}
\vspace{2pt}\\
Take any element $f \in \operatorname{rlp}(J_{cpt}) \cap W_{cpt}$.
The induced morphism $f_{K \subset L}$ is a fibration because $f$ belongs to $\operatorname{rlp}(J_{cpt})$.
In particular, $f_{\emptyset \subset K} : \mathcal{F}(K) \rightarrow \mathcal{G}(K)$ is a fibration.
Then, $f_K$ is a fibration by Proposition \ref{PropAugmented}.
Furthermore, $f_K$ is an acyclic fibration because $f$ belongs to $W_{cpt}$.
We consider the following diagram:
	\[\xymatrix{
		\mathcal{F}(L) \ar@/^20pt/[rrd]^-{f_L}_-{\simeq} \ar@/_20pt/[rdd] \ar[rd]^-{f_{K \subset L}}
	\\
		& \mathcal{F}(K) \times_{\mathcal{G}(K)} \mathcal{G}(L) \ar[r]^-{\simeq} \ar[d]
		& \mathcal{G}(L) \ar[d]
	\\
		& \mathcal{F}(K) \ar[r]_-{f_K}^-{\simeq}
		& \mathcal{G}(K).
	}\]
The above $\mathcal{F}(K) \times_{\mathcal{G}(K)} \mathcal{G}(L) \rightarrow \mathcal{G}(L)$ is an acyclic fibration because a pull-back of an acyclic fibration is an acyclic fibration.
In particular, this is a weak equivalence.
Then, $f_{K \subset L}$ is a weak equivalence.
Recall that $f_{K \subset L}$ is a fibration.
Then, $f_{K \subset L}$ is an acyclic fibration.
This means that $f$ belongs to $\operatorname{rlp}(I_{cpt})$.
Therefore, we obtain the inclusion relation $\operatorname{rlp}(I_{cpt}) \subset W_{cpt}$.
\vspace{2pt}\\
\underline{Proof of the condition {\bf (3)}}
\vspace{2pt}\\
This is obvious because the same property holds for weak equivalences in {\bf qTop}.
\vspace{2pt}\\
\underline{Proof of the condition {\bf (4)}}
\vspace{2pt}\\
We consider the following adjunction:
	\[\xymatrix{
		{\bf qTop} \ar@<-10pt>[rr]_-{X \mapsto X_{!K}} \ar@{}[rr] | -{\top}
		&& {\bf PSh}_\ast(\tilde{\mathscr{U}}_B; {\bf qTop}). \ar@<-10pt>[ll]_-{\mathcal{F}(K) \mapsfrom \mathcal{F}}
	}\]
The functor $\mathcal{F} \mapsto \mathcal{F}(K)$ preserves any colimit by the sectionwise calculation.
The functor $X \mapsto X_{!K}$ preserves any small object by Proposition \ref{PropSmallAdjunction}.
For any morphism $X \rightarrow Y$ belonging $I_{\bf qTop}$ and any morphism $X' \rightarrow Y'$ belonging $J_{\bf qTop}$, the objects $X$, $X'$, $Y$ and $Y'$ are small by Example \ref{ExSmallSheaf}.
The objects $X_{!K}$, $X'_{!K}$, $Y_{!K}$ and $Y'_{!K}$ are small for any compact subset $K$ of $B$.
Then, $X_{!L} \coprod_{X_{!K}} Y_{!K}$ and $X'_{!L} \coprod_{X'_{!K}} Y'_{!K}$ are small by Proposition \ref{PropSmallColimit}.
Therefore, $I_{cpt}$ and $J_{cpt}$ permit the small object argument.
\end{proof}
	\subsection{The pointwise projective model structure}
In this subsection, we will prove the following theorem.
\begin{thm}[the pointwise projective model structure] \label{ThmPointwiseProjective}
The category ${\bf PSh}_\ast(\tilde{\mathscr{U}}_B; {\bf qTop})$ has a model structure satisfying the followings:
\begin{itemize}
\item $f$ is a weak equivalence \\
$\Leftrightarrow$ $f$ is a pointwise weak equivalence
(i.e. $f_{\{ x \}}$ is a weak equivalence for any point $x \in B$, where $\{ x \}$ is a singleton).
\item $f$ is a fibration \\
$\Leftrightarrow$ $f$ is a pointwise fibration.
\end{itemize}
Furthermore, this model structure is right proper.
\end{thm}
To prove the above theorem, we apply Corollary \ref{CorTransfer}, the corollary of {\it Transfer Theorem} (cf. Theorem \ref{ThmTransfer}) of a cofibrantly generated model structure. \par
For each point $x \in B$, let ${\bf qTop}_x$ be a copy of the category {\bf qTop}.
There exist two sets $I_x$ and $J_x$ and a class $W_x$ as copies of the sets $I_{\bf qTop}$ and $J_{\bf qTop}$ and the class $W_{\bf qTop}$ respectively.
The category $\prod_x {\bf qTop}_x$ has a model structure.
Furthermore, the model category $\prod_x {\bf qTop}_x$ is cofibrantly generated.
We will explain this cofibrantly model structure.
For each point $x \in B$, we define two sets $\overline{I_x}$ and $\overline{J_x}$ of some morphisms in $\prod_x {\bf qTop}_x$ as the following:
\[
\begin{array}{lcl}
\overline{I_x} & := & \left(\displaystyle\prod_{y \neq x} \{ id_\emptyset \}\right) \times I_x, \\
\overline{J_x} & := & \left(\displaystyle\prod_{y \neq x} \{ id_\emptyset \}\right) \times J_x,
\end{array}
\]
where $\emptyset$ is the initial object in each category ${\bf qTop}_y$.
The model category $\prod_x {\bf qTop}_x$ is cofibrantly generated satisfying the followings:
\begin{itemize}
\item $\coprod_x \overline{I_x}$ generates cofibrations,
\item $\coprod_x \overline{J_x}$ generates acyclic cofibrations, and
\item $\prod_x W_x$ is the class of weak equivalences.
\end{itemize}\par
Next, we construct the following adjunction:
	\[\xymatrix{
		\displaystyle\prod_{x \in B} {\bf qTop} \ar@<-10pt>[rr]_-{F_{pt}} \ar@{}[rr] | -{\top}
		&& {\bf PSh}_\ast(\tilde{\mathscr{U}}_B; {\bf qTop}). \ar@<-10pt>[ll]_-{G_{pt}}
	}\]
We define the above two functors $F_{pt}$ and $G_{pt}$ as the followings:
\begin{itemize}
\item $G_{pt}(\mathcal{F}) := (\mathcal{F}(\{ x \}))_{x \in B}$, and
\item 
\(
[F_{pt}((X^x)_{x \in B})](U) := \left\{
	\begin{array}{ll}
	\emptyset & (|U| \ge 2) \\
	X^x & (U = \{ x \}) \\
	{\bf 1} & (U = \emptyset).
	\end{array}
\right.
\)
\end{itemize}
Then, there exists an adjunction $F_{pt} \dashv G_{pt}$. \par
We define two sets $I_{pt}$ and $J_{pt}$ and a class $W_{pt}$ of some morphisms in ${\bf PSh}_\ast(\tilde{\mathscr{U}}_B; {\bf qTop})$ as the followings:
\[
\begin{array}{lcl}
I_{pt} & := & F_{pt}(\coprod_x \overline{I_x}), \\
J_{pt} & := & F_{pt}(\coprod_x \overline{J_x}), \\
W_{pt} & := & \{ f \, | \, f_{\{ x \}} \mbox{ is a weak equivalence for }^\forall x \in B \}.
\end{array}
\]
Then, a morphism $g$ has the right lifting property with respect to $J_{pt}$ if and only if $G_{pt}(g)$ has the right lifting property with respect to $\coprod_x \overline{J_x}$.
This means that $g$ has the right lifting property with respect to $J_{pt}$ if and only if $g$ is a pointwise fibration.
\begin{proof}[Proof of Theorem \ref{ThmPointwiseProjective}]
In order to apply Corollary \ref{CorTransfer}, we show the followings:
\begin{description}
\item[(1)] $G_{pt}$ preserves any colimit.
\item[(2)] We have the inclusion relations $G_{pt}(F_{pt}(I)) \subset \operatorname{cell}(I_{pt})$ and $G_{pt}(F_{pt}(J_{pt})) \subset \operatorname{cell}(J_{pt})$.
\end{description}
The condition {\bf (1)} holds by the sectionwise calculation.
The condition {\bf (2)} holds by the isomorphism $G_{pt}F_{pt} \cong Id$.
Therefore, ${\bf PSh}_\ast(\tilde{\mathscr{U}}_B; {\bf qTop})$ has the cofibrantly generated model structure we want, by Corollary \ref{CorTransfer}.
This model structure is right proper by Proposition \ref{PropFibrantProper}.
\end{proof}
	\subsection{The pointwise flexible model structure}
In this subsection, we will prove the following theorem.
\begin{thm}[the pointwise flexible model structure] \label{ThmPointwiseFlexible}
The category ${\bf PSh}_\ast(\tilde{\mathscr{U}}_B; {\bf qTop})$ has a model structure satisfying the followings:
\begin{itemize}
\item $f$ is a weak equivalence \\
$\Leftrightarrow$ $f$ is a pointwise weak equivalence.
\item $f$ is a fibration \\
$\Leftrightarrow$ $f$ is a flexible extension.
\end{itemize}
Furthermore, this model structure is right proper.
\end{thm}
\begin{proof}
We construct the model structure we want by mixing the {\it compact flexible model structure} (cf. Theorem \ref{ThmCompactFlexible}) and the {\it pointwise projective model structure} (cf. Theorem \ref{ThmPointwiseProjective}).
In order to apply Theorem \ref{ThmMix}, we show that we have the following inclusion relations:
\begin{description}
\item[(1)] $W_{cpt} \subset W_{pt}$, and
\item[(2)] $\operatorname{rlp}(J_{cpt}) \subset \operatorname{rlp}(J_{pt})$.
\end{description}\par
For any point $x \in B$, the singleton $\{ x \}$ is compact.
Then, we obtain the inclusion relation $W_{cpt} \subset W_{pt}$.
If a morphism $f$ belongs to $\operatorname{rlp}(J_{cpt})$, $f_{\emptyset \subset K}$ is a fibration for any compact subset $K$ of $B$.
Then, $f_K$ is a fibration by Proposition \ref{PropAugmented}.
In particular, by considering the case where $K$ is a singleton $\{ x \}$, $f$ belongs to $\operatorname{rlp}(J_{pt})$.
We obtain the inclusion relation $\operatorname{rlp}(J_{cpt}) \subset \operatorname{rlp}(J_{pt})$.
Therefore, ${\bf PSh}_\ast(\tilde{\mathscr{U}}_B; {\bf qTop})$ has the model structure we want, by Theorem \ref{ThmMix}.
Furthermore, this model structure is right proper because the {\it pointwise projective model structure} is right proper
(cf. Theorem \ref{ThmPointwiseProjective} and Proposition \ref{PropMixProper}).
\end{proof}
	\section{Result 2 : The prefibration structure}
The goal of this section is to prove the following theorem.
\begin{thm} \label{Thm2}
Let $B$ be a strongly locally contractible normal Hausdorff space.
There exists an ABC prefibration structure on ${\bf Sh}(B; {\bf qTop})$ such that the followings hold. \par
Let $f : \mathcal{F} \rightarrow \mathcal{G}$ be a morphism in ${\bf Sh}(B; {\bf qTop})$.
\begin{itemize}
\item $f$ is a weak equivalence \\
$\Leftrightarrow$ $f$ is a stalkwise weak equivalence
(i.e. $f_x : \mathcal{F}_x \rightarrow \mathcal{G}_x$ is a weak equivalence for any $x \in B$).
\item $f$ is a fibration \\
$\Leftrightarrow$ $f$ is a flexible extension
(cf. Definition \ref{DefFlexible}).
\end{itemize}
In particular, for a continuous sheaf $\mathcal{F}$ on $B$, 
\begin{itemize}
\item $\mathcal{F}$ is fibrant \\
$\Leftrightarrow$ $\mathcal{F}$ is flexible
(cf. Definition \ref{DefFlexible}).
\end{itemize}
\end{thm}
The proof will be in Section 4.2.
The axioms of an ABC prefibration structure are a generalisation of the axioms of a model category.
Most of them on ${\bf Sh}(B; {\bf qTop})$ follow from Theorem \ref{Thm1}.
However, only the axiom 6. of Definition \ref{DefABC}, of the existence of a decomposition, needs to be proved independently.
To prove this, we define the {\it flexible track} and prove Lemma \ref{LemGromov2}, a generalisation of Lemma \ref{Gromov2}
(cf. Section 4.1).
\vspace{10pt}\\
Review the definition of an {\it ABC prefibration structure}.
\begin{defi}[A. Radulescu-Banu \cite{radulescu2006cofibrations}] \label{DefABC}
Let $\mathscr{C}$ be a category.
An ABC (Anderson-Brown-Cisinski) prefibration structure $(W, \operatorname{Fib})$ on $\mathscr{C}$ is a pair of two classes $W$ and $\operatorname{Fib}$ of some morphisms, called {\it weak equivalences} and {\it fibration} respectively, satisfying the following axioms:
\begin{enumerate}
\item $\mathscr{C}$ has a terminal object {\bf 1}.
(An object $X$ is {\it fibrant} if the unique morphism $X \rightarrow {\bf 1}$ is a fibration.)
\item $W$ and $\operatorname{Fib}$ are closed under composition.
\item Any isomorphism is a weak equivalence.
Furthermore, if the codomain is fibrant, any isomorphism is an acyclic fibration.
(A morphism $f$ is an {\it acyclic fibration} if $f$ is a fibration and a weak equivalence.)
\item If two of the three morphisms $f$, $g$ and $g \circ f$ are weak equivalences, then the third is also a weak equivalence.
\item Let $Y$ and $Y'$ be fibrant objects, and let $f : X \rightarrow Y$ be a fibration.
Then, there exists a pull-back $f' : X' \rightarrow Y'$ of $f$ along any morphism $Y' \rightarrow Y$.
And the morphism $f'$ is a fibration.
Furthermore, if $f$ is acyclic, then $f'$ is also acyclic.
\item Any morphism $f$ has a decomposition $f = p \circ e$ such that the morphism $e$ is a weak equivalence and the morphism $p$ is a fibration.
\end{enumerate}
\end{defi}
``ABC (Anderson-Brown-Cisinski)" comes from names of D. W. Anderson \cite{anderson1978}, K. S. Brown \cite{10.2307/1996573} and D.-C. Cisinski.
	\subsection{The flexible track}
Let $f : \mathcal{F} \rightarrow \mathcal{G}$ be a morphism between continuous sheaves on $B$.
$f$ induces a morphism $\tilde{f} : \mathcal{F}^\ast \rightarrow \mathcal{G}^\ast$ between the sheaves of formal sections.
We denote the mapping track of $\tilde{f}$ by $\tilde{f'} : \mathcal{F}^\ast_f \rightarrow \mathcal{G}^\ast$.
\begin{defi}
We call the above $\tilde{f'}$ the {\it flexible track} of $f$.
\end{defi}
\begin{rem}
If $\mathcal{G} = {\bf 1}$, then $\mathcal{G}^\ast = {\bf 1}$.
Then, we obtain the equality $\mathcal{F}^\ast = \mathcal{F}^\ast_{f}$.
\end{rem}
Recall that the claim {\bf (1)} of Lemma \ref{Gromov2} has been proved
(cf. Lemma \ref{LemGromov1}).
We can prove the claim {\bf (2)} of Lemma \ref{Gromov2}.
Here we prove a more general statement.
\begin{lem} \label{LemGromov2}
Let $B$ be a normal Hausdorff space.
For any morphism $f : \mathcal{F} \rightarrow \mathcal{G}$ between continuous sheaves on $B$, the flexible track $\tilde{f'}$ of $f$ is a flexible extension.
In prticular, for any continuous sheaf $\mathcal{F}$ on $B$, the sheaf $\mathcal{F}^\ast$ of formal sections of $\mathcal{F}$ is flexible.
\end{lem}
\begin{proof}
\setcounter{equation}{0}
For any two compact subsets $K$ and $L$ of $B$ with $K \subset L$, we will show that the morphism $\tilde{f'}_{K \subset L} : \mathcal{F}^\ast_f(L) \rightarrow \mathcal{F}^\ast_f(K) \times_{\mathcal{G}^\ast(K)} \mathcal{G}^\ast(L)$ is a fibration.
i.e. We will show that, for a following commutative square, there exists a lift $\gamma$:
\begin{equation}
\vcenter{\xymatrix{
		\Delta^{n-1} \ar[r]^-{\alpha} \ar[d]
		& \mathcal{F}^\ast_f(L) \ar[d]^-{\tilde{f'}_{K \subset L}}
	\\
		\Delta^{n-1} \times [0, 1]_1 \ar[r]_-{\beta} \ar@{.>}[ur]^-{\gamma}
		& \mathcal{F}^\ast_f(K) \times_{\mathcal{G}^\ast(K)} \mathcal{G}^\ast(L).
}}
\end{equation}
(The above $[0, 1]_1$ is a copy of the closed interval $[0, 1]$.
To distinguish it from the closed interval $[0, 1]_2$ which comes later, we mark it with a subscript.)
By the discussion in Section 1.1.4, there exist small enough open neighbourhoods $U$ and $V$ of $K$ and $L$ respectively with $U \subset V$ such that the following diagram is commutative:
\begin{equation}
\vcenter{\xymatrix{
		\Delta^{n-1} \ar[r]^-{\tilde{\alpha}} \ar[d]
		& \mathcal{F}^\ast_f(V) \ar[d]^-{\tilde{f'}_{U \subset V}}
	\\
		\Delta^{n-1} \times [0, 1]_1 \ar[r]_-{\tilde{\beta}}
		& \mathcal{F}^\ast_f(U) \times_{\mathcal{G}^\ast(U)} \mathcal{G}^\ast(V).
}}
\end{equation}
We have the equality $\mathcal{F}^\ast_f(V) = \mathcal{F}^\ast(V) \times_{\mathcal{G}^\ast(V)} \mathcal{G}^\ast(V)^{[0, 1]_2}$.
(The above $[0, 1]_2$ is a copy of the closed interval $[0, 1]$.
The above morphism $\tilde{\alpha} : \Delta^{n-1} \rightarrow \mathcal{F}^\ast_f(V)$ can be represented as $\tilde{\alpha} = \alpha^1 \times \alpha^2$ using the following $\alpha^1$ and $\alpha^2$:
\[
\begin{array}{l}
\alpha^1 : \Delta^{n-1} \rightarrow \mathcal{F}^\ast(V), \\
\alpha^2 : \Delta^{n-1} \rightarrow \mathcal{G}^\ast(V)^{[0, 1]_2}.
\end{array}
\]
To distinguish it from the closed interval $[0, 1]_1$ which appeared first, we mark it with a subscript.)
Similarly, we decompose the above morphism $\tilde{\beta}$.
$\tilde{\beta}$ can be represented as $\tilde{\beta} = \beta^1 \times \beta^2$ using the following $\beta^1$ and $\beta^2$:
\[
\begin{array}{l}
\beta^1 : \Delta^{n-1} \times [0, 1]_1 \rightarrow \mathcal{F}^\ast_f(U), \\
\beta^2 : \Delta^{n-1} \times [0, 1]_1 \rightarrow \mathcal{G}^\ast(V).
\end{array}
\]
Furthermore, $\beta^1$ can be represented as $\beta^1 = \beta^{11} \times \beta^{12}$ using the following $\beta^{11}$ and $\beta^{12}$:
\[
\begin{array}{l}
\beta^{11} : \Delta^{n-1} \times [0, 1]_1 \rightarrow \mathcal{F}^\ast(U), \\
\beta^{12} : \Delta^{n-1} \times [0, 1]_1 \rightarrow \mathcal{G}^\ast(U)^{[0, 1]_2}.
\end{array}
\]\par
Take a support function $\phi : V \rightarrow [0, 1]$ such that we have the inclusion relations $\operatorname{supp}(\phi) \subset U$ and the equality $\phi \equiv 1$ on an open neighbourhood of $K$.
(This is possible because $B$ is a normal Hausdorff space.)
Let $W := V - \operatorname{supp}(\phi)$. \par
Using these data $\phi$ and $W$ and the morphisms $\alpha^1$, $\alpha^2$, $\beta^{11}$, $\beta^{12}$ and $\beta^2$, we will construct the following morphisms $\gamma^1_1$, $\gamma^1_2$, $\gamma^2_1$ and $\gamma^2_2$:
\[
\begin{array}{l}
\gamma^1_1 : \Delta^{n-1} \times [0, 1]_1 \rightarrow \mathcal{F}^\ast(U), \\
\gamma^1_2 : \Delta^{n-1} \times [0, 1]_1 \rightarrow \mathcal{F}^\ast(W), \\
\gamma^2_1 : \Delta^{n-1} \times [0, 1]_1 \rightarrow \mathcal{G}^\ast(U)^{[0, 1]_2}, \\
\gamma^2_2 : \Delta^{n-1} \times [0, 1]_1 \rightarrow \mathcal{G}^\ast(W)^{[0, 1]_2}.
\end{array}
\]
Take the subset $\Sigma$ of $[0, 1]_1 \times [0, 1]_2$ and the strong deformation retraction $R_u : [0, 1]_1 \times [0, 1]_2 \rightarrow [0, 1]_1 \times [0, 1]_2$ in Lemma \ref{LemRetract}.
A strong deformation retraction $r_u : [0, 1]_1 \rightarrow [0, 1]_1$ of the inclusion $\{ 0 \} \hookrightarrow [0, 1]_1$ is defdined by the equality $R_u(s, 0) = (r_u(s), 0)$.
A continuous map $\psi : \Delta^{n-1} \times V \times [0, 1]_1 \rightarrow \Delta^{n-1} \times [0, 1]_1$ is defined as $\psi(p, x, s) := (p, r_{1-\phi(x)}(s))$.
Identifying $\Delta^{n-1} \times [0, 1]_1$ with $\Delta^n$ by a homeomorphism, we regard $\Delta^{n-1} \times [0, 1]_1$ as an object in ${\bf Sing\Delta}$.
By the discussion in Section 2.3, the continuous map $\psi$ induces a morphism
\[
\psi^\ast : \mathcal{F}^\ast[\Delta^{n-1} \times [0, 1]_1] |_{V} \rightarrow \mathcal{F}^\ast[\Delta^{n-1} \times [0, 1]_1] |_{V}
\]
between sheaves of sets. \par
Let $\gamma^1_1 := \psi^\ast(\beta^{11})$.
Let $\gamma^1_2 := \alpha^1 |_W \circ \operatorname{pr}$, where the morphism $\alpha^1 |_W$ is the composite morphism $\Delta^{n-1} \overset{\alpha^1}{\rightarrow} \mathcal{F}^\ast(V) \rightarrow \mathcal{F}^\ast(W)$ and $\operatorname{pr}$ is the canonical projection $\Delta^{n-1} \times W \times [0, 1]_1 \rightarrow \Delta^{n-1} \times W$. \par
A continuous map $\Psi : \Delta^{n-1} \times V \times [0, 1]_1 \times [0, 1]_2 \rightarrow \Delta^{n-1} \times [0, 1]_1 \times [0, 1]_2$ is defined as $\Psi(p, x, s, t) := (p, R_{1-\phi(x)}(s, t))$.
By the discussion in Section 2.3, the continuous map $\Psi$ induces a morphism
\[
\Psi^\ast : \mathcal{G}^\ast[\Delta^{n-1} \times [0, 1]_1 \times [0, 1]_2] |_{V} \rightarrow \mathcal{G}^\ast[\Delta^{n-1} \times [0, 1]_1 \times [0, 1]_2] |_{V}
\]
between sheaves of sets. \par
Let $\gamma^2_1 := \Psi^\ast(\beta^{12})$.
To define $\gamma^2_2$, we define $\nu_1$ and $\nu_2$ as the following compositions:
\[
\begin{array}{l}
\nu_1 : \Delta^{n-1} \times \{ 0 \} \times [0, 1]_2 \overset{\cong}{\rightarrow} \Delta^{n-1} \times [0, 1]_2 \overset{\alpha^2}{\rightarrow} \mathcal{G}^\ast(W), \\
\nu_2 : \Delta^{n-1} \times [0, 1]_1 \times \{ 1 \} \overset{\cong}{\rightarrow} \Delta^{n-1} \times [0, 1]_1 \overset{\beta^2}{\rightarrow} \mathcal{G}^\ast(W).
\end{array}
\]
We have the equality $\nu_1 = \nu_2$ on $\Delta^{n-1} \times \{ 0 \} \times \{ 1 \}$ because the diagram (2) is commutative.
Then, we obtain a morphism $\nu : \Delta^{n-1} \times W \times \Sigma \rightarrow \mathcal{G}^\ast(W)$ by pasting $\nu_1$ and $\nu_2$.
Let $\gamma^2_2 := \nu \circ (id \times R_1)$. \par
$\gamma^1_1$ and $\gamma^2_1$ induce a morphism $\gamma_1 : \Delta^{n-1} \times [0, 1]_1 \rightarrow \mathcal{F}^\ast_f(U)$.
$\gamma^1_2$ and $\gamma^2_2$ induce a morphism $\gamma_2 : \Delta^{n-1} \times [0, 1]_1 \rightarrow \mathcal{F}^\ast_f(W)$.
$\gamma_1$ and $\gamma_2$ induce a morphism $\tilde{\gamma} : \Delta^{n-1} \times [0, 1]_1 \rightarrow \mathcal{F}^\ast_f(V)$ with the sheaf condition of $\mathcal{F}^\ast_f$ with respect to the open covering $\{ U, W \}$ of $V$.
$\tilde{\gamma}$ induces a lift $\gamma$ of the diagram (1).
Therefore, the morphism $\tilde{f'}_{K \subset L}$ has the right lifting property with respect to $J_{\bf qTop}$.
This means that $\tilde{f'}$ is a flexible extension.
\end{proof}
	\subsection{A prefibration structure for continuous sheaves}
We can prove Theorem \ref{Thm2}.
\begin{proof}[proof of Theorem \ref{Thm2}]
We will check the axioms {\it 1.} to {\it 6.} of Definition \ref{DefABC}.
The axioms {\it 1.} to {\it 5.} is obvious by Theorem \ref{Thm1}.
We will check the axiom {\it 6.}. \par
We consider the following diagram on ${\bf Sh}(B; {\bf qTop})$, embedded in ${\bf PSh}_\ast(\tilde{\mathscr{U}}_B; {\bf qTop})$:
	\[\xymatrix{
		& \mathcal{F}^\ast_f \ar[dr]^-{\tilde{f'}}
	\\
		\mathcal{F}^\ast \ar[ur]^-{\simeq} \ar[rr]
		&& \mathcal{G}^\ast
	\\
		& \mathcal{F}^\ast_f \times_{\mathcal{G}^\ast} \mathcal{G} \ar[uu] | \hole \ar[dr]
	\\
		\mathcal{F} \ar[uu]^-{\simeq} \ar[rr]_-{f} \ar[ur]
		&& \mathcal{G}. \ar[uu]_-{\simeq}
	}\]
The diagonal morphisms $\mathcal{F} \rightarrow \mathcal{F}^\ast$ and $\mathcal{G} \rightarrow \mathcal{G}^\ast$ are stalkwise weak equivalences by Lemma \ref{LemGromov1}.
The flexible track $\tilde{f'}$ of $f$ is a flexible extension by Lemma \ref{LemGromov2}.
The above morphism $\mathcal{F}^\ast \rightarrow \mathcal{F}^\ast_f$ is a sectionwise weak equivalence.
Then, this is a stalkwise weak equivalence by Proposition \ref{PropSectionwiseStalkwise}.
The above morphism $\mathcal{F}^\ast_f \times_{\mathcal{G}^\ast} \mathcal{G} \rightarrow \mathcal{F}^\ast_f$ is a stalkwise weak equivalence because the model category ${\bf PSh}_\ast(\tilde{\mathscr{U}}_B; {\bf qTop})$ is right proper.
Then, the above morphism $\mathcal{F} \rightarrow \mathcal{F}^\ast_f \times_{\mathcal{G}^\ast} \mathcal{G}$ is a stalkwise weak equivalence.
The above morphism $\mathcal{F}^\ast_f \times_{\mathcal{G}^\ast} \mathcal{G} \rightarrow \mathcal{G}$ is a flexible extension because it is the pull-back of a flexible extension.
Then, the decomposition $\mathcal{F} \rightarrow \mathcal{F}^\ast_f \times_{\mathcal{G}^\ast} \mathcal{G} \rightarrow \mathcal{G}$ of $f$ is what we want.
The axiom {\it 6.} is satisfied.
Therefore, the category ${\bf Sh}(B; {\bf qTop})$ of continuous sheaves has an ABC prefibration structure we want.
\end{proof}
\appendix
	\section{Model categories}
A {\it model category} is a category in which Homotopy Theory can be discussed.
In this paper, basic concepts such as a model category, a Quillen adjunction, etc. are assumed to be known.
See \cite{hovey2007model} or \cite{hirschhorn2009model} for more information.
We review some of the additional properties.
\vspace{10pt}\\
Note the following notation used in the model category theory.
\begin{defi} \label{DefNotation}
\mbox{}
\begin{itemize}
\item $f \pitchfork g$ $\Leftrightarrow$ $g$ has the right lifting property with respect to $f$.
\end{itemize}
For a class $I$ of some morphisms, we define the followings.
\begin{itemize}
\item $\operatorname{rlp}(I)$ is the class of morphisms having the right lifting property with respect to $I$.
\item $\operatorname{cell}(I)$ is the class of relative $I$-cell complexes.
\item $\operatorname{Cof}(I)$ is the class of retracts of relative $I$-cell complexes.
\end{itemize}
\end{defi}
	\subsection{Cofibrantly generated model categories}
It is difficult to check the axioms of a model category.
It is easier to check that the necessary and sufficient conditions of a {\it cofibrantly generated} model category.
	\subsubsection{Small object argument}
\begin{defi}[small objects]
Let $\mathscr{C}$ be a cocomplete category, and let $\mathscr{D}$ be a subcategory of $\mathscr{C}$.
\begin{itemize}
\item Let $\kappa$ be a cardinal.
An object $W$ in $\mathcal{C}$ is {\it $\kappa$-small relative to $\mathscr{D}$} if, for any regular cardinal $\lambda (\ge \kappa)$ and any $\lambda$-sequence $X$ in $\mathscr{D}$, the following map is a bijection:
\[
\displaystyle\lim_{\substack{\longrightarrow \\ \beta}} \mathscr{C}(W, X_\beta) \cong \mathscr{C}(W, \displaystyle\lim_{\substack{\longrightarrow \\ \beta}} X_\beta).
\]
(Remark that $\displaystyle\lim_{\longrightarrow} X$ is a colimit in $\mathscr{C}$.)
\item In particular, if $\mathscr{D} = \mathscr{C}$, $W$ is {\it $\kappa$-small}.
\item An object $W$ in $\mathcal{C}$ is {\it small relative to $\mathscr{D}$} if there exists a cardinal $\kappa$ such that the object $W$ is $\kappa$-small relative to $\mathscr{D}$.
\item In particular, if $\mathscr{D} = \mathscr{C}$, $W$ is {\it small}.
\end{itemize}
\end{defi}
\begin{ex} \label{ExSmallPresheaf}
Let $\mathscr{U}$ be a small category.
For any object $u \in Ob(\mathscr{U})$, the presheaf $\mathscr{U}(-, u) \in {\bf PSh}(\mathscr{U}; {\bf Set})$ is $\aleph_0$-small.
In fact, by Yoneda's Lemma, there exists the following bijection:
\[
\begin{array}{rcl}
\displaystyle\lim_{\substack{\longrightarrow \\ \beta}} {\bf PSh}(\mathscr{U}; {\bf Set})(\mathscr{U}(-, u), X_\beta)
& \cong & \displaystyle\lim_{\substack{\longrightarrow \\ \beta}} (X_\beta [u]) \\
& \cong & (\displaystyle\lim_{\substack{\longrightarrow \\ \beta}} X_\beta) [u] \\
& \cong & {\bf PSh}(\mathscr{U}; {\bf Set})(\mathscr{U}(-, u), \displaystyle\lim_{\substack{\longrightarrow \\ \beta}} X_\beta).
\end{array}
\]
\end{ex}
\begin{ex} \label{ExSmallSheaf}
Let $\mathscr{U}$ be a small site such that every covering has a finite refinement.
If a sheaf $\mathcal{F}$ on $\mathscr{U}$ is $\kappa$-small as a presheaf, the sheaf $\mathcal{F}$ is $\kappa$-small as a sheaf.
There exists the following bijection:
\[
\begin{array}{rcl}
\displaystyle\lim_{\substack{\longrightarrow \\ \beta}} {\bf Sh}(\mathscr{U}; {\bf Set})(\mathcal{F}, X_\beta)
& \cong & \displaystyle\lim_{\substack{\longrightarrow \\ \beta}} {\bf PSh}(\mathscr{U}; {\bf Set})(\mathcal{F}, X_\beta) \\
& \cong & {\bf PSh}(\mathscr{U}; {\bf Set})(\mathscr{U}(-, u), \displaystyle\lim_{\substack{\longrightarrow \\ \beta}} X_\beta) \\
& \cong & {\bf Sh}(\mathscr{U}; {\bf Set})(\mathscr{U}(-, u), \displaystyle\lim_{\substack{\longrightarrow \\ \beta}} X_\beta).
\end{array}
\]
(Remark that the colimit $\displaystyle\lim_{\substack{\longrightarrow}} X_\beta$ as a sheaf is the colimit as a presheaf because every covering has a finite refinement.)
\end{ex}
\begin{prop}[\mbox{cf. \cite[Proposition 10.4.8]{hirschhorn2009model}}] \label{PropSmallColimit}
Let $W : \mathscr{P} \rightarrow \mathscr{C}$ be a small diagram
(i.e. $W$ is a functor from a finite category $\mathscr{P}$).
If each $W_p$($p \in \operatorname{Ob}(\mathscr{P})$) is small relative to $\mathscr{D}$, the colimit $\displaystyle\lim_{\substack{\longrightarrow}} W_p$ is also small relative to $\mathscr{D}$.
\end{prop}
\begin{prop} \label{PropSmallAdjunction}
Let $F \dashv G : \mathscr{C} \rightarrow \mathscr{C'}$ be an adjunction, and let $\mathscr{D}$ be a subcategory of $\mathscr{C}$.
Suppose that the functor $G$ preserves any filtered colimit.
If an object $c \in Ob(\mathscr{C})$ is $\kappa$-small relative to $\mathscr{D}$, the object $F(c) \in Ob(\mathscr{C})$ is $\kappa$-small relative to $G^{-1}(\mathscr{D})$.
\end{prop}
\begin{proof}
\[
\begin{array}{rcl}
\displaystyle\lim_{\substack{\longrightarrow \\ \beta}} \mathscr{C'}(F(c), X_\beta)
& \cong & \displaystyle\lim_{\substack{\longrightarrow \\ \beta}} \mathscr{C}(c, G(X_\beta)) \\
& \cong & \mathscr{C}(c, \displaystyle\lim_{\substack{\longrightarrow \\ \beta}} G(X_\beta)) \\
& \cong & \mathscr{C}(c, G(\displaystyle\lim_{\substack{\longrightarrow \\ \beta}} X_\beta)) \\
& \cong & \mathscr{C'}(F(c), \displaystyle\lim_{\substack{\longrightarrow \\ \beta}} X_\beta).
\end{array}
\]
\end{proof}
\begin{defi}[small object argument] \label{DefSmallArgument}
Let $\mathscr{C}$ be a cocomplete category, and let $I$ be a set of some morphisms in $\mathscr{C}$.
\begin{itemize}
\item Let $\kappa$ be a cardinal.
$I$ {\it permits the $\kappa$-small object argument} if, for any morphism $x \rightarrow y$ belonging $I$, the domain $x$ is $\kappa$-small relative to the subcategory $\operatorname{cell}(I)$.
\item $I$ {\it permits the small object argument} if there exists a cardinal $\kappa$ such that $I$ permits the $\kappa$-small object argument.
\end{itemize}
\end{defi}
	\subsubsection{Cofibrantly generated model categories}
\begin{defi}[cofibrantly generated model categories] \label{DefCofibrantlyGenerated}
A model category $\mathscr{C}$ is {\it cofibrantly generated} if there exist two sets $I$ and $J$ of some morphisms in $\mathscr{C}$ satisfying the following:
\begin{itemize}
\item The class of cofibrations of $\mathscr{C}$ is $\operatorname{Cof}(I)$
(cf. Definition \ref{DefNotation}).
\item The class of acyclic cofibrations of $\mathscr{C}$ is $\operatorname{Cof}(J)$
(cf. Definition \ref{DefNotation}).
\item $I$ and $J$ permit the small object argument
(cf. Definition \ref{DefSmallArgument}).
\end{itemize}
Then, we say that $I$ (resp. $J$) generates cofibrations (resp. acyclic cofibrations).
\end{defi}
\begin{thm}[\mbox{Recognition Theorem; cf. \cite[Theorem 11.3.1]{hirschhorn2009model}}] \label{ThmRecognition}
Let $\mathscr{C}$ be a complete and cocomplete category.
Let $W$ be a class of some morphisms in $\mathscr{C}$, and let $I$ and $J$ be two sets of some morphisms in $\mathscr{C}$.
There exists a model structure on $\mathscr{C}$ such that the class of weak equivalences is $W$, the class of cofibrations is generated by $I$ and the class of acyclic cofibrations is generated by $J$ if and only if the following conditions hold:
\begin{enumerate}
\item We have the inclusion relation $\operatorname{Cof}(J) \subset \operatorname{Cof}(I) \cap W$
(cf. Definition \ref{DefNotation}).
\item We have the inclusion relation $\operatorname{rlp}(I) \subset \operatorname{rlp}(J) \cap W$
(cf. Definition \ref{DefNotation}).
\item The reverse inclusion holds for one of the two inclusions above.
i.e. We have one of the following inclusion relations:
	\begin{enumerate}
	\item $\operatorname{Cof}(J) \supset \operatorname{Cof}(I) \cap W$, or
	\item $\operatorname{rlp}(I) \supset \operatorname{rlp}(J) \cap W$.
	\end{enumerate}
\item If two of the three morphisms $f$, $g$ and $g \circ f$ are contained in $W$, then the third is also contained in $W$.
\item $I$ and $J$ permit the small object argument
(cf. Definition \ref{DefSmallArgument}).
\end{enumerate}
\end{thm}
\begin{thm}[\mbox{Transfer Theorem; cf. \cite[Theorem 11.3.2]{hirschhorn2009model}}] \label{ThmTransfer}
Let $\mathscr{C}$ be a cofibrantly generated model category such that a set $I$ generates cofibrations and that a set $J$ generates acyclic cofibrations.
Let $W$ be the class of weak equivalences in $\mathscr{C}$.
For a category $\mathscr{D}$ and adjunction $F \dashv G : \mathscr{C} \rightarrow \mathscr{D}$, suppose the followings:
\begin{enumerate}
\item $F(I)$ and $F(J)$ permit the small object argument
(cf. Definition \ref{DefSmallArgument}).
\item We have the inclusion relation $\operatorname{cell}(F(J)) \subset G^{-1}(W)$
(cf. Definition \ref{DefNotation}).
\end{enumerate}
Then, $\mathscr{D}$ has a cofibrantly generated model structure such that the class of weak equivalences is $G^{-1}(W)$, the class of cofibrations is generated by $F(I)$ and the class of acyclic cofibrations is generated by $F(J)$.
Furthermore, the adjunction $F \dashv G$ is a Quillen adjunction.
\end{thm}
Proofs using Theorem \ref{ThmTransfer} may be simplified if we can use the following corollary.
\begin{cor} \label{CorTransfer}
Let $\mathscr{C}$, $I$, $J$ and $W$ be the same as in Theorem \ref{ThmTransfer}.
For a category $\mathscr{D}$ and adjunction $F \dashv G : \mathscr{C} \rightarrow \mathscr{D}$, suppose the followings:
\begin{enumerate}
\item $G$ preserves any colimit.
\item We have the inclusion relations $G(F(I)) \subset \operatorname{cell}(I)$ and $G(F(J)) \subset \operatorname{cell}(J)$
(cf. Definition \ref{DefNotation}).
\end{enumerate}
Then, $\mathscr{D}$ has a cofibrantly generated model structure such that the class of weak equivalences is $G^{-1}(W)$, the class of cofibrations is generated by $F(I)$ and the class of acyclic cofibrations is generated by $F(J)$.
Furthermore, the adjunction $F \dashv G$ is a Quillen adjunction.
\end{cor}
\begin{proof}
In order to apply Theorem \ref{ThmTransfer}, we show the followings:
\begin{description}
\item[(1)] $F(I)$ and $F(J)$ permit the small object argument.
\item[(2)] We have the inclusion relation $\operatorname{cell}(F(J)) \subset G^{-1}(W)$.
\end{description}\par
We will show the condition {\bf (1)}.
$I$(resp. $J$) permits the small object argument.
i.e. There exists a cardinal $\kappa$ such that, for any morphism $x \rightarrow y$ belonging $I$ (resp. $J$), the domain $x$ is $\kappa$-small relative to the subcategory $\operatorname{cell}(I)$ (resp. $\operatorname{cell}(J)$).
From assumption {\it 1.} and Proposition \ref{PropSmallAdjunction}, the domain of any morphism belonging to $F(I)$ (resp. $F(J)$) is $\kappa$-small relative to $G^{-1}(\operatorname{cell}(I))$ (resp. $G^{-1}(\operatorname{cell}(J))$).
From assumption {\it 1.}, we obtain the inclusion relation $G(\operatorname{cell}(F(I))) \subset \operatorname{cell}(G(F(I)))$ (resp. $G(\operatorname{cell}(F(J))) \subset \operatorname{cell}(G(F(J)))$).
From assumption {\it 2.}, we obtain the inclusion relation $\operatorname{cell}(G(F(I))) \subset \operatorname{cell}(I)$ (resp. $\operatorname{cell}(G(F(J))) \subset \operatorname{cell}(J)$).
Then, we have the inclusion relation $\operatorname{cell}(F(I)) \subset G^{-1}(\operatorname{cell}(I))$ (resp. $\operatorname{cell}(F(J)) \subset G^{-1}(\operatorname{cell}(J))$).
Therefore, the domain of any morphism belonging to $F(I)$ (resp. $F(J)$) is $\kappa$-small relative to $\operatorname{cell}(F(I))$ (resp. $\operatorname{cell}(F(J))$).
Then, $F(I)$ and $F(J)$ permit the small object argument.\par
We will show the condition {\bf (2)}.
From assumption {\it 1.}, we obtain the inclusion relation $G(\operatorname{cell}(F(J))) \subset \operatorname{cell}(G(F(J)))$.
From assumption {\it 2.}, we obtain the inclusion relation $\operatorname{cell}(G(F(J))) \subset \operatorname{cell}(J)$.
Then, we have the sequence of the inclusion relations $G(\operatorname{cell}(F(J))) \subset \operatorname{cell}(G(F(J))) \subset \operatorname{cell}(J) \subset \operatorname{Cof}(J) \subset W$.
Therefore, we have the inclusion relation $\operatorname{cell}(F(J)) \subset G^{-1}(W)$.
\end{proof}
	\subsection{Cartesian closed model categories}
Prepare some notations.
\begin{defi} \label{DefProductSum}
Let $\mathscr{C}$ be a complete and cocomplete category.
For two morphisms $f : A \rightarrow B$ and $f' : A' \rightarrow B'$ in $\mathscr{C}$, denote the following morphism $(A \times B') \cup_{A \times A'} (B \times A') \rightarrow B \times B'$ by $f \triangledown f'$:
	\[\xymatrix{
		&& B \times B'
	\\
		A \times B' \ar[r] \ar@/^15pt/[urr]^-{f \times B'}
		& (A \times B') \cup_{A \times A'} (B \times A') \ar[ur]^-{f \triangledown f'}
	\\
		A \times A' \ar[r]^-{f \times A'} \ar[u]^-{A \times f'}
		& B \times A'. \ar[u] \ar@/_15pt/[uur]_-{B \times f'}
	}\]
Furthermore, suppose that $\mathscr{C}$ is Cartesian closed.
For two morphisms $f : A \rightarrow B$ and $g : X \rightarrow Y$ in $\mathscr{C}$, denote the following morphism $X^B \rightarrow Y^B \times_{Y^A} X^A$ by $g^{\# f}$:
	\[\xymatrix{
		X^B \ar@/_15pt/[ddr]_-{g^B} \ar@/^15pt/[drr]^-{X^f} \ar[dr]^{g^{\# f}}
	\\
		& Y^B \times_{Y^A} X^A \ar[d] \ar[r]
		& X^A \ar[d]^-{g^A}
	\\
		& Y^B \ar[r]^-{Y^f}
		& Y^A.
	}\]
\end{defi}
We define a Cartesian closed model category.
\begin{defi}[Cartesian closed model categories] \label{DefModelCartesian}
A {\it Cartesian closed model category} is a model category that is Cartesian closed, satisfying any (i.e. all) of the following equivalent conditions:
\begin{itemize}
\item For any two cofibrations $f$ and $f'$, $f \triangledown f'$ is a cofibration.
In addition, if $f$ or $f'$ is acyclic, then $f \triangledown f'$ is also acyclic.
\item For any cofibration $f$ and any fibration $g$, $g^{\# f}$ is a fibration.
In addition, if $f$ or $g$ is acyclic, then $g^{\# f}$ is also acyclic.
\end{itemize}
\end{defi}
\begin{prop} \label{PropCofibrantlyCartesian}
Let $\mathscr{C}$ be a model category which is Cartesian closed.
Suppose that $\mathscr{C}$ is a cofibrantly generated model category such that a set $I$ generates cofibrations and that a set $J$ generates acyclic cofibrations.
The following conditions are equivalent.
\begin{enumerate}
\item $\mathscr{C}$ is a Cartesian closed model category.
\item For any two elements $f, f' \in I$, $f \triangledown f'$ is a cofibration.
For any element $f \in I$ and any element $f' \in J$, $f \triangledown f'$ is an acyclic cofibration.
\end{enumerate}
\end{prop}
\begin{proof}
{\it 1.} $\Rightarrow$ {\it 2.} is obvious.
We will show {\it 2.} $\Rightarrow$ {\it 1.}. \par
By {\it 2.}, For any two elements $f, f' \in I$ and any acyclic fibration $g$, we have the relation $f \triangledown f' \pitchfork g$.
This is equivalent to the relation $f \pitchfork g^{\# f'}$.
This means that $g^{\# f'}$ is an acyclic fibration.
Then, for any cofibration $f$, we have the relation $f \pitchfork g^{\# f'}$.
This is equivalent to the relation $f \triangledown f' \pitchfork g$.
In the same way, $f'$ can be any cofibration. 
Therefore, for any two cofibrations $f$ and $f'$, $f \triangledown f'$ is a cofibration. \par
We discuss the same for any $f \in I$ and $f' \in J$, and any fibration $g$.
Then, for any cofibration $f$ and any acyclic cofibration $f'$, $f \triangledown f'$ is an acyclic cofibration.
Therefore, $\mathscr{C}$ is a Cartesian closed model category.
\end{proof}
	\subsection{Proper model categories}
\begin{defi}
A model category $\mathscr{C}$ is {\it right proper} if any pullback of a weak equivalence along a fibration becomes a weak equivalence.
\end{defi}
\begin{prop}[\mbox{cf. \cite[Corollary 13.1.3]{hirschhorn2009model}}] \label{PropFibrantProper}
If all objects in a model category $\mathscr{C}$ are fibrant, then $\mathscr{C}$ is right proper.
\end{prop}
	\subsection{Mixed model categories}
\begin{thm}[M. Cole \cite{cole2006mixing}] \label{ThmMix}
Let $\mathscr{C}$ be a (complete and cocomplete) category.
Let $(W, \operatorname{Fib}, \operatorname{Cof})$ and $(W', \operatorname{Fib'}, \operatorname{Cof'})$ be two model structures on $\mathscr{C}$.
Suppose the followings:
\begin{itemize}
\item $W \subset W'$,
\item $\operatorname{Fib} \subset \operatorname{Fib'}$.
\end{itemize}
Then, there exists a model structure on $\mathscr{C}$ such that the class of weak equivalences is $W'$ and that the class of fibration is $\operatorname{Fib}$.
\end{thm}
\begin{prop}[\mbox{M. Cole \cite[Proposition 4.1]{cole2006mixing}}] \label{PropMixProper}
In the situation of Theorem \ref{ThmMix}, if the model structure $(W', \operatorname{Fib'}, \operatorname{Cof'})$ is right proper, then the mixed model structure is also right proper.
\end{prop}

\bibliography{h-principle, complex, homotopy, geo_and_top, category, homology}
\bibliographystyle{plain}


\end{document}